\documentclass[]{amsart}

\usepackage{latexsym}
\usepackage{amssymb,amsmath,amsopn}
\usepackage[dvips]{graphicx}   %To insert ps figures
\usepackage{color,epsfig}      %To insert ps figures
\usepackage{url}
\usepackage{amsbsy}

\newtheorem{theorem}{Theorem}[]
\newtheorem{lemma}{Lemma}[]
\newtheorem{statement}{Statement}[]

\theoremstyle{definition}

\newtheorem{remark}[]{Remark}[]
\newtheorem{example}[]{Example}[]

\title{Proper quadrics in the Euclidean $n$-space }
\author[\'A. G.Horv\'ath]{\'Akos G.Horv\'ath}
\address {Department of Geometry \\
Budapest University of Technology and Economics\\
H-1521 Budapest\\
Hungary}
\email{ghorvath@math.bme.hu}

\dedicatory{}
\subjclass[2010]{51N20}
\keywords{analytical Euclidean geometry, confocal quadrics, construction of Chasles, Staude's wire model}

\date{July, 2017}

\begin{document}

\begin{abstract}
In this paper we investigate the metric properties of quadrics and cones of the $n$-dimensional Euclidean space. As applications of our formulas we give a more detailed description of the construction of Chasles and the wire model of Staude, respectively.
\end{abstract}

\maketitle

\section{Introduction}

There are several books and papers either on conics or on quadrics of the $3$-space. The synthetic properties of these are in the common knowledge of mathematics, physics and several other researchers of experimental sciences for a long time. A sort of nice books has written in the nineteenth century on the two and three dimensional cases, respectively. In the contemporary literature we can find nice and more general projective building up of quadrics however (unfortunately) without a systematic investigation of those metric properties which are well-known in the three dimensional case.  As an example I would like to mention here the most known book of Berger \cite{berger} in which we can find two chapters (Chapter 14 and Chapter 15) on the projective and affine properties of quadrics and only one really $n$-dimensional metric theorem (the Theorem of Apollonius in the paragraph 15.6.2). The same situation can be found in the most paper, the metric geometry of proper $n$-dimensional quadric had not been investigated. Using the standard analytic geometry of the $n$-space we would like to fill this gap in the literature.

As applications of our formulas we give a more detailed description of the construction of Chasles and the wire model of Staude, respectively.

\section{Quadrics in $\mathbb{R}^n$}

\subsection{The definition of quadrics} A non-degenerated central proper surface of second order is called by \emph{quadric} in this paper. In canonical form it can be written as the set of points $x=(x_1,\ldots, x_n)^T$ holding the equality
\begin{equation}\label{eq:quadric}
\sum\limits_{i=1}^n\varepsilon_i\frac{x_i^2}{a_i^2}=1,
\end{equation}
with $a_n\leq a_{n-1}\leq \ldots \leq a_1$ and $\varepsilon_i\in\{\pm 1\}$. Clearly, if $x'$ is any point of the quadric above a normal vector at $x'$ is equal to
\begin{equation}\label{eq:normal}
n(x')=\left(\varepsilon_1\frac{x'_1}{a_1^2},\ldots, \varepsilon_n\frac{x'_n}{a_n^2}\right)^T \mbox{ with norm square } |n(x')|^2=\sum\limits_{i=1}^{n}\frac{(x'_i)^2}{a_i^4}
\end{equation}
and its tangent hyperplane at this point is the set of points $x$ with
$$
0=\langle n(x'),x-x'\rangle=\sum\limits_{i=1}^{n}\varepsilon_i\frac{x_ix'_i-{x'_i}^2}{a_i^2}=\sum\limits_{i=1}^{n}\varepsilon_i\frac{x_ix'_i}{a_i^2}-1
$$
We use the name \emph{ellipsoid} for a quadric when $\varepsilon_i=1$ hold for all $i$. (In the case when $\varepsilon=-1$ for all $i$ the quadric is the empty set.)

\subsection{Conjugacy with respect to an ellipsoid and the theorem of Apollonius}

Two diameters of a given quadric are \emph{conjugate diameters}, if they are conjugate diameters of that $2$-dimensional conic which is the intersection of the quadric by the $2$-plane is spanned by the diameters. (Two directions are conjugate to each other if they are parallels to a pair of conjugate diameters.) This means that the line parallel to the first diameter and going through the end point of the second one is a tangent line of the quadric at the investigated point. (In this section we consider ellipsoids and all the corresponding signs are plus.) The analytic condition of conjugacy can be determined on the following way. Let denote $e=(e_1,\ldots, e_n)$ and $f=(f_1,\ldots,f_n)$ the investigated directions. They are conjugate to each other and let $p_e$ be that real number for which  the point $p_ee$ lie on the quadric of form (\ref{eq:quadric}). Then we have
\begin{equation}\label{eq:conjugatediam}
0=\langle n(p_ee),f\rangle=p_e\sum\limits_{i=1}^{n}\frac{e_if_i}{a_i^2}.
\end{equation}
A \emph{complete system of conjugate semi-diameters} $\{x^1,\ldots,x^n\}$ by definition has $n$ pairwise conjugate elements. For an ellipsoid the above conditions can be collected in a matrix form, too. Let denote by $X$, $X_A$ and $A$ the matrices
$$
X=\left(
  \begin{array}{ccc}
    x^1_1 & \cdots & x^n_1 \\
    \vdots & \cdots & \vdots \\
    x^1_n & \cdots & x^n_n \\
  \end{array}
\right), \quad
X_A=\left(
      \begin{array}{ccc}
        \frac{x^1_1}{a_1} & \cdots & \frac{x^1_n}{a_n} \\
        \vdots & \cdots & \vdots \\
        \frac{x^n_1}{a_1} & \cdots & \frac{x^n_n}{a_n} \\
      \end{array}
    \right)
\mbox{ and }
A=\left(
    \begin{array}{ccccc}
      a_1 & 0 & \cdots & 0 & 0\\
      0 & a_2 & 0 &\cdots & 0 \\
      0 & 0& a_3 & \cdots & 0 \\
      \vdots & \vdots &\vdots & \vdots &\vdots \\
      0 & 0 & \cdots & 0 & a_n \\
    \end{array}
  \right),
$$
respectively.
Clearly $X_A\cdot A=X$ and $X_A$ by the above assumption is an orthogonal matrix. From this follows that
$$
X^TX=\left(X_A\cdot A\right)^T\left(X_A\cdot A\right)^T=A^TA=A^2
$$
implying that the respective traces of the two sides are equal to each other. This means that
\begin{equation}\label{eq:squaresumofthediam}
|x^1|^2+|x^2|^2+\ldots +|x^n|^2=a_1^2+\ldots +a_n^2.
\end{equation}
We get that \emph{the sum of squares of a complete system of conjugate semi-diameters is a constant}. The orthogonality of the matrix $X_A$ immediately implies that its determinant is $1$, hence we also have:
\begin{equation}\label{eq:volumeofpar}
\det X =\det (X_A\cdot A)=\det A=a_1\cdots a_n,
\end{equation}
meaning that \emph{the volume of the parallelepiped is spanned by the semi-diameters of a complete conjugate system is a constant}.

The above two observations are special cases of the following theorem:

\begin{theorem}[Theorem of Apollonius]\label{thm:apollonius}
An ellipsoid in the Euclidean $n$-space has $n$ associated scalars invariants with respect to any affinity of the space. It can be defined by the formulas:
\begin{equation}\label{eq:invariants}
v_k:=\sum\limits_{i_1<i_2<\cdots <i_k}\det G[x^{i_1},\cdots,x^{i_k}],
\end{equation}
where $\{x^1,\ldots,x^n\}$ is a complete conjugate system of semi-diameters and $G[x^{i_1},\cdots,x^{i_k}]$ is the Gram matrix of the vector system $x^{i_1},\cdots,x^{i_k}$.
\end{theorem}

\begin{proof} Consider the matrix $X^TX+tI$ where $t$ is a real number and $I$ is the identity matrix. Then we have
$$
\det\left[(X^TX+tI)\right]=\det[A^2+tI]=\sum\limits_{k=0}^nt^{n-k}\prod_{i_1<\ldots<i_k}a_{i_1}^2\cdots a_{i_k}^2.
$$
On the other hand, if $X^TX=:G=[g_{i,j}]$ is the Gram matrix of $X$, and $G_{i_1,\ldots,i_k}$ denotes its $k$-minor defined by the elements in the crossing of those rows and columns which correspond to the multi-index $\{i_1,\ldots,i_k\}$, than using an inductive argument and the rule on the addition of determinants we can prove that
$$
\det\left[(G+tI)\right]=\sum\limits_{^k=0}^n t^{n-k} \sum\limits_{i_i<\ldots <i_k}\det[G_{i_1,\ldots,i_k}].
$$
These two formulas imply that for all sets of complete conjugate systems hold the equality
$$
v_k=\prod_{i_1<\ldots<i_k}a_{i_1}^2\cdots a_{i_k}^2
$$
which gives our statement.
\end{proof}

Two points $x$ and $y$ are \emph{conjugate with respect to an ellipsoid} if the directions $x$ and $y-x$ are conjugate to each other. This means by condition (\ref{eq:conjugatediam}) that the equality  $x^TA^{-2}y=1$ hold with the diagonal matrix $A^{-2}$. Clearly, it is a symmetric relation.
Let $\langle H, x\rangle=\sum\limits_{i=1}^nh_ix_i=1$ is a given hyperplane does not contain the origin. The \emph{pole} of this hyperplane with respect to the ellipsoid is such a point of the space which is conjugate to an arbitrary point of the hyperplane. If $\xi$ is its pole with respect to the ellipsoid then for a point $x\in H$ we have two equations for all $x\in H$:
$$
\langle H, x\rangle=1 \mbox{ and } 1=x^TA^{-2}\xi=\sum\limits_{i=1}^n \frac{x_i\xi_i}{a_i^2}.
$$
Now immediately can be calculated the coordinates of the pole:
\begin{equation}\label{eq:poleofaplane}
\xi_i=h_ia_i^2 \quad i=1,\ldots, n.
\end{equation}

\subsection{The system of confocal quadrics}

Consider an ellipsoid and its confocals in the common canonical form defined by the equalities
\begin{equation}\label{eq:confocals}
\sum\limits_{i=1}^n\frac{x_i^2}{a_i^2-\lambda}=1,
\end{equation}
where we have $0<a_n<a_{n-1}<\ldots <a_1$ and $\lambda\in\mathbb{R}$. We denote by $\mathcal{C}(\lambda)$ this system of quadric and $C(\lambda)$ the quadric corresponding to the parameter $\lambda$, respectively.

Clearly, if $\lambda=a_k^2$ then we should consider only such points as a point of ${C}(\lambda)$, for which the assumption $x_k=0$ holds and the other coordinates satisfy the equation
\begin{equation}\label{eq:focalqudric}
\sum\limits_{\substack{i=1 \\ i\ne k}}^n\frac{x_i^2}{a_i^2-a_k^2}=1.
\end{equation}
Hence the above $n-1$-dimensional quadric can be considered as the limit of the pencil of confocal quadrics by the assumption $\lambda \rightarrow a_k$. Denote by $C_k$ this surface which we call the \emph{$k$-th focal quadric} of the system of quadric $\mathcal{C}(\lambda)$. There are precisely $n-1$ focal quadric with real points and one of them (when $\lambda=a_1^2$) does not contain real points.
It is clear that the $n$-th focal quadric $C_n$ is an ellipsoid.
First we detail the basic properties of the confocal quadrics.

\subsubsection{Confocal quadrics through a given point of an ellipsoid} Observe that a point $x'=(x'_1,\ldots, x'_n)^T$ of the space belongs to a system of confocal (to the given one) quadrics of cardinality $n$ with pair-wise orthogonal elements. In fact, we can arrange the original equation to the form
\begin{equation}\label{eq:flambdafunc}
0=\sum\limits_{i=1}^n{x'_i}^2\prod\limits_{j\ne i}(a_j^2-\lambda)-\prod\limits_{i=1}^n(a_i^2-\lambda)=:f(\lambda).
\end{equation}
Since $f(-\infty)=-\infty ,\, f(a_n)>0, \, f(a_{n-1})<0, \ldots $ by Role's theorem there are $n$ distinct roots $\lambda^j$ of $f(\lambda)$ for the various $\lambda$. Every solutions define a system of semi-axes denoted by $\{a_i^j \, i=1,\ldots n\}$ where $j=1,\ldots ,n$ and the corresponding system of signs $\varepsilon(i,j)$ before the squared length of the semi-axis $a_i^j$. (If we choose a point on the ellipsoid determining the confocal system  we assume that $a_i^1:=a_i$). For brevity, in this paper  the values $\varepsilon(i,j)(a_i^j)^2$ we denote by $(a_i^j)^2$, our notation contains that sign which have to occur in the canonical form of the corresponding quadric. Observe that by the fixed distance of the squared lengths of the corresponding semi-axes in a confocal system, these are create a monotone decreasing sequence for a fixed $j$. On the other hand we also assume that the values $a_1^j$ gives a monotone decreasing sequence in $j$, namely the inequalities
$(a_1^1)^2 > (a_1^2)^2> \ldots > (a_1^n)^2$ are hold, too.

By these squared-lengths we can determine the coordinates of the common point $x'$. Let denote by $h_i^2:=a_1^2-a_i^2$ for $i>1$ and $h_1=0$, and search the possible major semi-axis $a'_1$ of a confocal quadric through $x'$ from the polynomial equation of order $n$:
$$
\sum\limits_{i=1}^n\frac{{x'_i}^2}{{a'_1}^2-h_i^2}=1,
$$
or equivalently
\begin{equation}\label{eq:confonapoint}
0=\left({a'_1}^2\right)^n-\left({a'_1}^2\right)^{n-1}\left(\sum\limits_{i=1}^n{x'_i}^2+\sum\limits_{i=1}^nh_i^2\right)\pm \ldots (-1)^n{x'_1}^2\prod\limits_{i=2}^nh_i^2
\end{equation}
Let denote by $\left(a_1^j\right)^2$ the $j$-th root of the equation (\ref{eq:confonapoint}). It is always positive (we have an ellipsoid solution). (Analogously, by parity of reasoning, we might have taken $a'_i$ for our unknown and in this case the notation $\left(a_i^j\right)^2$ denote the corresponding root of the examined equation independently from the fact that it is positive or negative. If the root is negative we consider the corresponding term in the canonical form of the quadric with negative sign.) Using Vieta's formulas we get
$$
{x'_1}^2=\frac{\prod\limits_{j=1}^n\left(a_1^j\right)^2}{\prod\limits_{i=2}^nh_i^2}=\frac{\prod\limits_{j=1}^n\left(a_1^j\right)^2} {\prod\limits_{i=2}^n\left(a_1^2-a_i^2\right)}=\frac{\prod\limits_{j=1}^n\left(a_1^j\right)^2} {\prod\limits_{i=2}^n\left(\left(a_1^1\right)^2-\left(a_i^1\right)^2\right)}
$$
and generally
\begin{equation}\label{eq:coordwithconfoc}
{x'_i}^2=\frac{\prod\limits_{j=1}^n\left(a_i^j\right)^2} {\prod\limits_{\substack{i=1 \\ i\ne j}}^n\left(\left(a_j^j\right)^2-\left(a_i^j\right)^2\right)} \mbox{ for all } i.
\end{equation}
The second term in (\ref{eq:confonapoint}) is the sum of the roots of the equation. From this immediately follows the connection
$$
\left(\sum\limits_{i=1}^n{x'_i}^2+\sum\limits_{i=1}^nh_i^2\right)=\sum\limits_{j=1}^n\left(a_1^j\right)^2,
$$
implying again an important equality:
\begin{equation}\label{eq:lengthbytheconfaxes}
|x'|^2=\sum\limits_{i=1}^n{x'_i}^2=\sum\limits_{j=1}^n\left(a_j^j\right)^2,
\end{equation}
since  $\left(a_1^j\right)^2-\left(a_1^2-a_j^2\right)=\left(a_1^j\right)^2-\left(\left(a_1^j\right)^2-\left(a_j^j\right)^2\right)=\left(a_j^j\right)^2$.

Observe that \emph{the confocal quadrics through the given point $x'$ are pairwise orthogonal at this point}. We have to prove that for $j\ne k$
$$
0=\sum\limits_{i=1}^{n}\frac{\left(x'_i\right)^2}{\left(a_i^j\right)^2\left(a_i^k\right)^2}.
$$
But we have
$$
1=\sum\limits_{i=1}^{n}\frac{\left(x'_i\right)^2}{\left(a_i^j\right)^2}=\sum\limits_{i=1}^{n}\frac{\left(x'_i\right)^2}{\left(a_i^k\right)^2},
$$
if we subtract one of these equations from the other, we get
\begin{equation}\label{eq:onorthog}
0=\sum\limits_{i=1}^{n}\left(\left(a_i^k\right)^2-\left(a_i^j\right)^2\right)\frac{\left(x'_i\right)^2}{\left(a_i^j\right)^2\left(a_i^k\right)^2}= \left(\left(a_1^k\right)^2-\left(a_1^j\right)^2\right)\sum\limits_{i=1}^{n}\frac{\left(x'_i\right)^2}{\left(a_i^j\right)^2\left(a_i^k\right)^2}
\end{equation}
equivalently with our statement. Therefore where $n$ confocal quadrics intersect, each tangent hyperplane cuts the other perpendicularly, and the tangent hyperplane to any one contains the normals to the other $n-1$.

\subsubsection{On the central sections of a quadric}\label{sssec:centrals}

First we prove that if a plane be drawn through the centre parallel to any tangent hyperplane to a quadric, the axes of the section made by that plane are parallel to the normals to the $n-1$ confocals through the point of contact. It has been proved that the latter directions are pairwise orthogonal to each other, it remains to prove only that these directions also conjugate directions with respect to the intersection quadric. Using (\ref{eq:normal}) and the assumption (\ref{eq:conjugatediam}) we have to prove that for $j\ne k$ holds
$$
0=\sum\limits_{i=1}^{n}\frac{\left(x'_i\right)^2}{\left(a_i^1\right)^2\left(a_i^j\right)^2\left(a_i^k\right)^2}= \sum\limits_{i=1}^{n}\frac{1}{\left(\left(a_i^k\right)^2-\left(a_i^j\right)^2\right)}\left(\frac{\left(x'_i\right)^2}{\left(a_i^1\right)^2\left(a_i^j\right)^2}- \frac{\left(x'_i\right)^2}{\left(a_i^1\right)^2\left(a_i^k\right)^2}\right).
$$
But for all $i$ the equality $\left(\left(a_1^k\right)^2-\left(a_1^j\right)^2\right)=\left(\left(a_i^k\right)^2-\left(a_i^j\right)^2\right)$ is hold thence the right hand side is equal to the subtraction of the two equalities of form (\ref{eq:onorthog}).

It can also be determined \emph{the lengths of the squared semi-axes of the central section} parallel to the tangent hyperplane at the point $x'$. From (\ref{eq:quadric}) we get that the length $\rho^1_j$ of the radius vector of the given quadric (which is the first confocal ones) parallel to the normal of the $j$-th confocal  holds the equality
$$
1=\frac{\left(\rho^1_j\right)^2}{|n_j(x')|^2}\sum\limits_{i=1}^{n} \frac{(x'_i)^2}{\left(a_i^1\right)^2\left(a_i^j\right)^4}=\left(\rho^1_j\right)^2 \frac{\sum\limits_{i=1}^{n} \frac{(x'_i)^2}{\left(a_i^1\right)^2\left(a_i^j\right)^4}}{\sum\limits_{i=1}^{n} \frac{(x'_i)^2}{\left(a_i^j\right)^4}}=\left(\rho^1_j\right)^2 \frac{\sum\limits_{i=1}^{n} \frac{1}{\left(\left(a_i^1\right)^2-\left(a_i^j\right)^2\right)} \left(\frac{(x'_i)^2}{\left(a_i^j\right)^4}-\frac{(x'_i)^2}{\left(a_i^1\right)^2\left(a_i^j\right)^2}\right)}{\sum\limits_{i=1}^{n} \frac{(x'_i)^2}{\left(a_i^j\right)^4}}.
$$
But the value $\left(\left(a_i^1\right)^2-\left(a_i^j\right)^2\right)$ is independent from the lower index $i$ and we saw in (\ref{eq:onorthog}) that for all $i$ and $j$ holds
$$
0=\sum\limits_{i=1}^{n}\frac{\left(x'_i\right)^2}{\left(a_i^j\right)^2\left(a_i^k\right)^2},
$$
we get that for all $j\geq 2$ the squared length of the $j$-th semi-axis is equal to:
\begin{equation}\label{eq:radius}
\left(\rho^1_j\right)^2=\left(a_1^1\right)^2-\left(a_1^j\right)^2.
\end{equation}

\subsubsection{Duality of confocal systems}\label{sssec:duality}

Denote by $(p^j)^2$ the reciprocal of the norm square of the $j$-th normal vector $n_j(x')$ at $x'$. Geometrically $p^j$ is the length of the foot-point of the perpendicular from the origin to the tangent hyperplane of the $j^{th}$ quadric. In fact,
$$
\left\langle \frac{p^j}{|n_j(x')|}n_j(x')-x',n_j(x')\right\rangle=p^j|n_j(x')|-\langle x',n_j(x')\rangle=1-\sum\limits_{i=1}^n\frac{(x'_i)^2}{\left(a_i^j\right)^2}=0.
$$
Using the observation on the volume of the parallelepiped is spend by conjugate semi-diameters proved in the first section, we have that for all $j$
\begin{equation}
\left(p^j\right)^2\cdot \prod\limits_{k\ne j}\left(\rho^j_k\right)^2=\prod\limits_{i=1}^n\left(a_i^j\right)^2.
\end{equation}
We hold again the opportunity that we involve the sign of the squared value in the notation (as at (\ref{eq:coordwithconfoc})), we can write that for all $j$ we have
\begin{equation}\label{eq:onpj}
\left(p^j\right)^2=\frac{\prod\limits_{i=1}^n\left(a_i^j\right)^2}{\prod\limits_{k\ne j}\left(\left(a_1^j\right)^2-\left(a_1^k\right)^2\right)}
\end{equation}
defining a formula analogous to that of (\ref{eq:coordwithconfoc}). It can be observed the symmetry which exists between these values for $p^i$, and the values already found in (\ref{eq:coordwithconfoc}) for $x'_i$. If the $n$ tangent hyperplanes had been taken as coordinate hyperplanes, $p^i$ would be the coordinates (with suitable sign) of the centre $p$ of the surface. So we have a duality theorem:
\begin{statement}\label{st:duality}
With the point $x'$ as the centre $n$ confocals may be described having the $n$ tangent hyperplanes for principal planes and intersecting in the centre of the original system of surfaces. The axes of the new system of confocals are $a^1_1,a^2_1,\ldots,a^n_1$, $a^1_2,a^2_2,\ldots a^n_2$ $\ldots $ $a^1_n,a^2_n,\ldots a^n_n$. The $n$ tangent hyperplanes of the new (dual) system are the $n$ principal planes of the original system.
\end{statement}

\subsubsection{Polarity with respect to a confocal system}

We introduced the concept of conjugate points earlier. In this paragraph we prove some statements on the impact of the polarity to a confocal system of quadrics.
\begin{statement}\label{st:poleofahypplane}
The locus of the pole of a given hyperplane $\langle H,x\rangle=1$ with regard to a system of confocal surfaces a right line perpendicular to the given plane.
\end{statement}
\begin{proof}
Using (\ref{eq:poleofaplane}) we get the equations $\xi_i=h_i\left(a_i^2-\lambda\right)$ for $i=1,\ldots ,n$. From this we have
$$
-\lambda=\frac{x_1}{h_1}-a_1^2=\cdots =\frac{x_n}{h_n}-a_n^2
$$
which is the system of equations of a line perpendicular to the given hyperplane.
\end{proof}

\begin{remark}
From this we can construct a surface confocal with a given one and touching a given plane. In fact, the pole of the plane is a point of the above line which is perpendicular to the hyperplane. This line intersects the hyperplane in a point which is a point of the searched confocal.
\end{remark}

\begin{remark}
Another consequence of Statement \ref{st:poleofahypplane} that the locus of the pole of a tangent hyperplane of any quadric with respect to any confocal is the normal of the quadric in the considered point of tangency.
\end{remark}

\begin{statement}\label{st:sevpartangent}
If several parallel tangent hyperplanes touch a series of confocal quadrics, the locus of their points of contact is an equilateral hyperbola.
\end{statement}

\begin{proof}
Consider now a system of parallel hyperplanes and determine the locus of points of tangency with respect to a given system of confocal quadrics. By the above remark the searched points are the intersection points of parallel lines through the poles of the hyperplanes by the corresponding perpendicular hyperplanes, respectively. Since parallel hyperplanes have a common $n-2$-dimensional ideal subspace, hence their poles are collinear and lying on the diameter $d$ of the original quadric conjugate to the hyperplane through the origin and belonging to the given pencil of parallel hyperplanes. This means that the points in question are on the plane determined by the diameter $d$ and the direction $n$ of the normals of the given parallel hyperplanes. Let $\Pi$ any hyperplane with equation $\sum_{i}h_ix_i=1$ then by the equation (\ref{eq:poleofaplane}) the coordinates of its pole are $\xi_i=h_ia_i^2$, where the semi-axes of the given quadrics are $a_i$ for $i=1,\ldots, n$. The coordinates of the point $P$ are $p_i=h_ia_i^2+th_i$, where $\sum_ip_ih_i=1$. From this we get that $t=(1-\sum_ih_i^2a_i^2)/\sum_ih_i^2$, and the distance of the point $P$ and the hyperplane through the origin and parallel to $\Pi$ is
$$
d(P,X)=\left|\frac{\langle OP,h\rangle}{|h|^2}h\right|=\frac{\sum_ih_i^2a_i^2+\frac{1-\sum_ih_i^2a_i^2}{\sum_ih_i^2}\sum_ih_i^2}{\sqrt{\sum_ih_i^2}}=\frac{1}{\sqrt{\sum_ih_i^2}}
$$
On the other hand the distance of the point $P$ from the line through the origin with direction vector $h=(h_1,\ldots,h_n)^T$ is
$$
d(P,Y)=\sqrt{|OX|^2-\left(\frac{\left|\langle OX,h\rangle\right|}{|h|^2}|h|\right)^2}=\sqrt{\sum_ih_i^2a_i^4-\frac{\left(\sum_ih_i^2a_i^2\right)^2}{\sum_ih_i^2}}= \frac{\sqrt{\sum_ih_i^2\sum_ih_i^2a_i^4-\left(\sum_ih_i^2a_i^2\right)^2}}{\sqrt{\sum_ih_i^2}}.
$$
The production of the two quantity is
$$
F(h)=\sqrt{\sum_ih_i^2a_i^4-\frac{\left(\sum_ih_i^2a_i^2\right)^2}{\sum_ih_i^2}}= \frac{\sqrt{\sum_ih_i^2\sum_ih_i^2a_i^4-\left(\sum_ih_i^2a_i^2\right)^2}}{\sum_ih_i^2}.
$$
Since for all $\tau\in\mathbb{R}$ $F(h)=F(\tau h)$ thus with respect to the Descartian coordinate system with axes $OX$ and $OY$ the production of the coordinates of $P$ are the same constant, hence the locus of the point $P$ is an equilateral hyperbola with orthogonal asymptotes $OX$ and $OY$.
\end{proof}

\begin{statement}[The Apollonian pedal curve]\label{st:apollonianhyp}
Consider the set of homothetic ellipsoids $E(\lambda)=\{x\in \mathbb{R}^n : \quad \sum_{i=1}^n \left(x_i/\lambda a_i\right)^2=1\}$ $\lambda\in\mathbb{R}^+$ and a point $P=(u_1,\ldots u_n)^T$ of the space. The locus of those points of the ellipsoids which are nearest to $P$ is an equilateral hyperbolic curve of order $n$ which we call \emph{Apollonian curve} of the system of the homothetic system.
\end{statement}

\begin{proof}
For a fixed real $\lambda $ we can determine those points of the ellipsoid $E(\lambda)$ in which the normal lines of the surface go through the point $P$. Hence with a real $t$ has to hold the vector equation $t n(x)=u-x$ (or equivalently $x=u-tn(x)$). Thus we get that for the unknown values $x_i$ have
$$
u_i-x_i=t\frac{x_i}{\left(\lambda a_i\right)^2} \quad \mbox{ for }\quad i=1,\ldots, n.
$$
Rearranging the equations we get that
$$
(t=)\frac{\left(\lambda a_1\right)^2\left(u_1-x_1\right)}{x_1}=\frac{\left(\lambda a_2\right)^2\left(u_2-x_2\right)}{x_2}=\ldots =\frac{\left(\lambda a_n\right)^2\left(u_n-x_n\right)}{x_n}.
$$

This leads to the system of equations
\begin{eqnarray}
\nonumber a_1^2\left(u_1-x_1\right)x_2 &= &a_2^2\left(u_2-x_2\right)x_1 \\
\nonumber a_1^2\left(u_1-x_1\right)x_3 &= &a_3^2\left(u_2-x_2\right)x_1 \\
\nonumber  & \vdots & \\
\nonumber a_1^2\left(u_1-x_1\right)x_n &= &a_n^2\left(u_n-x_n\right)x_1
\end{eqnarray}
which can be simplified to the following one
\begin{eqnarray}\label{eq:apolloniancurve}
\nonumber 0 & = & (a_1^2-a_2^2)x_1x_2 + a_2^2u_2 x_1-a_1^2u_1x_2 \\
& \vdots & \\
\nonumber 0 & = & (a_1^2-a_n^2)x_1x_n + a_n^2u_n x_1-a_1^2u_1x_n.
\end{eqnarray}
We can assume that $a_1^2$ is not equal to $a_i^2$ for $i>1$, because for concentric spheres the searched locus is known it is the line through the common center. Using the substitutions
$$
x_1=y_1, \, x_2=y_2-\frac{a_2^2u_2}{(a_1^2-a_2^2)}, \, \ldots \,,x_n=y_n-\frac{a_n^2u_n}{(a_1^2-a_n^2)},
$$
we get the form
\begin{eqnarray}
\nonumber 0 & = & (a_1^2-a_2^2)y_1y_2 -a_1^2u_1y_2+\frac{a_1^2u_1a_2^2u_2}{(a_1^2-a_2^2)} \\
\nonumber &\vdots & \\
\nonumber 0 & = & (a_1^2-a_2^2)y_1y_n -a_1^2u_1y_n+\frac{a_1^2u_1a_n^2u_n}{(a_1^2-a_n^2)},
\end{eqnarray}
and introducing the parameter $a_1^2u_1\tau:=y_1$, the parametric representation of the curve $A_n(\tau)$ with respect to the coordinate system $\{y_i\}$ is:
\begin{equation}\label{eq:apolloniuscurveinpar}
A_n(\tau):\tau \mapsto \left(a_1^2u_1\tau,\frac{-a_2^2u_2}{(a_1^2-a_2^2)\left((a_1^2-a_2^2)\tau-1\right)},\ldots ,\frac{-a_n^2u_n}{(a_1^2-a_n^2)\left((a_1^2-a_n^2)\tau-1\right)}\right)^T.
\end{equation}
Hence $A_n(\tau)$ is a rational algebraic curve of order $n$ (see \cite{sommerville}), independent from the parameter $\lambda $. It is the intersection of hyperbolic right cylinders based on the equilateral hyperbolas of the $2$-planes, $(y_1,y_2)$, $(y_1,y_3)$, \ldots ,$(y_1,y_n)$, respectively. The poles of the curve are at the parameter values
$$
\tau=\infty, \quad  \tau=\frac{1}{(a_1^2-a_2^2)}, \quad \tau=\frac{1}{(a_1^2-a_3^2)},\quad \ldots \quad \tau=\quad \frac{1}{(a_1^2-a_n^2)}
$$
respectively. Hence the asymptotes of the curve are parallel to the axes $x_1$,\ldots $x_n$ and go through the respective points
$$
p_1=\left(\begin{array}{c}0 \\
0 \\
\vdots \\
0
\end{array}\right),\,
p_2=\left(
\begin{array}{c}
\frac{a_1^2u_1}{(a_1^2-a_2^2)} \\
 0 \\
 \frac{-a_3^2u_3(a_1^2-a_2^2)}{(a_1^2-a_3^2)(a_2^2-a_3^2)}\\
 \vdots \\
 \frac{-a_n^2u_n(a_1^2-a_2^2)}{(a_1^2-a_n^2)(a_2^2-a_n^2)}
\end{array}\right), \,
p_3=\left(\begin{array}{c}
\frac{a_1^2u_1}{(a_1^2-a_3^2)} \\
\frac{-a_2^2u_2(a_1^2-a_3^2)}{(a_1^2-a_2^2)(a_3^2-a_2^2)} \\
0 \\
\vdots \\
\frac{-a_n^2u_n(a_1^2-a_3^2)}{(a_1^2-a_n^2)(a_3^2-a_n^2)}
\end{array}
\right), \,
 \cdots \,
p_n=\left(\begin{array}{c}
\frac{a_1^2u_1}{(a_1^2-a_n^2)} \\
\frac{-a_2^2u_2(a_1^2-a_n^2)}{(a_1^2-a_2^2)(a_n^2-a_2^2)} \\
\vdots \\
\frac{-a_{n-1}^2u_{n-1}(a_1^2-a_{n}^2)}{(a_1^2-a_{n-1}^2)(a_n^2-a_{n-1}^2)} \\
0
\end{array}
\right)
$$
where the coordinates are calculated with respect to the translated basis $\{y_i\}$. The asymptotes are of the form $y=p_i+se_i$ where $e_i$ is the unit direction vector of the axis $y_i$ (or equivalently of the axis $x_i$).
\end{proof}

\begin{remark}\label{rem:apollonianhyp}
If $n=2$ the above Apollonian curve is the so-called Apollonian equilateral hyperbola. As we know from the book \cite{berger} it is also the locus of the centers of the elements of that pencil of conics which determined by the given ellipse ${E}$ and a fixed circle which center is $P$ and intersects ${E}$ in four points. An immediate proof of this fact is the following. If a point $Q$ of the plane is a center of a conic of the pencil, then its polar lines with respect to the given ellipse and circle are parallel to each other. In fact, for any point $Q$ there is a point $Q'$ in the projective plane that the polar of $Q$ with respect to the element of the pencil are go through $Q'$. Since the polar of $Q$ is the ideal line if it is a center of a conic from the pencil, hence the point $Q'$ is also an ideal one. Since the polar of a point with respect to a circle is perpendicular to the segment connecting the point with the center of the circle this common direction of the two polar lines is known. Hence the locus does not depends on the radius of the circle, it only depends on the ellipse and the given point.
Consider now such a center $Q$ which is on the ellipse ${E}$. The polar $q$ of $Q$ with respect to ${E}$ is the tangent of $E$ at $Q$ hence it is perpendicular to $QP$ showing that $Q$ is on the Apollonian hyperbola $\mathcal{A}_2(\tau)$. Clearly, $\mathcal{A}_2(\tau)$ contains also the center of ${E}$ and the point $P$.
We prove that the locus is a conic implying that it is agree the Apollonian hyperbola as we stated. In fact, if $Q'$ and $R'$ are the ideal points corresponding to the above connection to the centers $Q$ and $R$, respectively, the polars of $Q'$ and $R'$ with respect to the conics in the pencil go through $Q$ and $R$, respectively. If $S'$ any further (ideal) point of the line $Q'R'$ then its polar $s'_Q$ with respect to the conic with center $Q$ and its polar $s'_R$ with respect to the conic with center $R$ intersect in a point $S$ which is conjugate to $S'$. Of course $S$ is a center of a conic of the pencil. In this way, we got two pencils of lines with support $Q$ and $R$ are in a projective (but not perspective) connection and the corresponding elements intersect each other in the points $S$ of the searched locus, hence the locus by the Steiner's definition of conics is a conic.
Observe that those ideal points which are giving by the directions of the axes of the ellipse are points of the locus, too since their polar lines with respect to the ellipse and the circle are parallel to each other. Hence the above conic is such a hyperbola which has (at least) five common points with the Apollonian hyperbola implying that the two conics are agree.
\end{remark}

\section{Proper cones of second order}

Another important class of surfaces of second order the class of \emph{proper cones}.  We define by the so-called canonical form the cones which apex at the origin with the equation
\begin{equation}\label{eq:cone}
\sum\limits_{i=1}^n\varepsilon_i\frac{x_i^2}{a_i^2}=0,
\end{equation}
where $a_n\leq a_{n-1}\leq \ldots \leq a_1$ and $\varepsilon_i\in\{\pm 1\}$. The \emph{semi-axes} of this cone are the pairwise orthogonal segments on the corresponding coordinate axes with respective lengths $a_1,\ldots,a_n$. Clearly the origin is the only point of a cone if the signs $\varepsilon_i$ are equal to each other.  In the other cases, the cones contain lines through the origin. The intersection of a proper cone with an affine hyperplane (is not going through the origin) is a quadric of dimension $n-1$. For brevity, use the simple notation $(a_i)^2$ for the signed and squared $i$-th semi-axis, meaning $(a_i)^2=a_i^2$ if $\varepsilon_i=1$ and $(a_i)^2=-a_i^2$ if $\varepsilon=-1$. Then we also write
$$
\sum\limits_{i=1}^n\frac{x_i^2}{(a_i)^2}=0
$$
for the above canonical form.
\subsection{On the intersection of proper confocal cones with common vertex.}

The system of proper confocal cones with vertex at the origin is defined by the equality
\begin{equation}\label{eq:confocalcones}
\sum\limits_{i=1}^n\frac{x_i^2}{a_i^2-\lambda}=0,
\end{equation}
where we have $0<a_n<a_{n-1}<\ldots <a_1$ and $\lambda\in\mathbb{R}$.
Our problem is to determine the common edges of such cones. In the $n$-dimensional space we consider $n-1$ confocal cones whose canonical equations are of the form
$$
\sum\limits_{i=1}^n\frac{x_i^2}{(\alpha^k_i)^2}=0 \quad k=1,\ldots, n-1,
$$
respectively\footnote{One of our purposes to determine the common edges that cones with common vertex $x'$ which go through the respective focal quadrics of an ellipsoid. As we will see in Statement \ref{st:axesoffocalcones} in this case $(\alpha^k_i)^2=(a^{i}_{k+1})^2$, where $(a^{i}_{k+1})^2$ is the signed and squared $k+1$-th axes of the $i$-th confocal through the point $x'$.}.   Then for their intersection we have a system of equations which we can compute.
It can be considered as a linear homogenies system of equations with respect to the variable $y_i=x_i^2$ with the matrix
$$
B=\left(
  \begin{array}{cccc}
    \frac{1}{(\alpha^1_1)^2} & \frac{1}{(\alpha^2_1)^2} & \ldots & \frac{1}{(\alpha^{n-1}_1)^2} \\
    \frac{1}{(\alpha^1_2)^2} & \frac{1}{(\alpha^2_2)^2} & \ldots & \frac{1}{(\alpha^{n-1}_2)^2} \\
    \vdots & \vdots & \vdots & \vdots \\
    \frac{1}{(\alpha^1_{n-1})^2} & \frac{1}{(\alpha^2_{n-1})^2} & \ldots & \frac{1}{(\alpha^{n-1}_{n-1})^2} \\
    \frac{1}{(\alpha^1_n)^2} & \frac{1}{(\alpha^2_{n})^2} & \ldots & \frac{1}{(\alpha^{n-1}_n)^2} \\
  \end{array}
\right).
$$
Every solutions are in the form $t\cdot N$, where $N\in \mathbb{R}^n$ is a vector orthogonal to the columns of $B$. Denote by $B_i$ the $(n-1)\times (n-1)$ matrix can be get from $B$ omitting its $i$-th row. Clearly we can choose for normal vector
$$
N=:\left(\det B_1,\ldots ,(-1)^{i-1}\det B_i,\ldots , (-1)^{n-1}\det B_n\right).
$$
The confocality property implies that for a given $l>1$ we have
$$
\frac{1}{(\alpha^l_i)^2}-\frac{1}{(\alpha^1_i)^2}=\frac{1}{(\alpha^l_i)^2}-\frac{1}{(\alpha^l_i)^2+\lambda(l)}= \frac{-\lambda(l)}{(\alpha^l_i)^2(\alpha^1_i)^2}
$$
holds for all $i$ with a system of $\lambda(l)$ independent from $i$. Hence the determinant of $B_k$ is equal to
$$
\det B_k=\det \left(
  \begin{array}{cccc}
    \frac{1}{(\alpha^1_1)^2} & \frac{-\lambda(2)}{(\alpha^2_1)^2(\alpha^1_1)^2} & \ldots & \frac{-\lambda(n-1)}{(\alpha^{n-1}_1)^2(\alpha^1_1)^2} \\
    \vdots & \vdots & \vdots & \vdots \\
    \frac{1}{(\alpha^1_{k-1})^2} & \frac{-\lambda(2)}{(\alpha^2_{k-1})^2(\alpha^1_{k-1})^2} & \ldots & \frac{-\lambda(n-1)}{(\alpha^{n-1}_{k-1})^2(\alpha^1_{k-1})^2} \\
    \frac{1}{(\alpha^1_{k+1})^2} & \frac{-\lambda(2)}{(\alpha^2_{k+1})^2(\alpha^1_{k+1})^2} & \ldots & \frac{-\lambda(n-1)}{(\alpha^{n-1}_{k+1})^2(\alpha^1_{k+1})^2} \\
    \vdots & \vdots & \vdots & \vdots \\
    \frac{1}{(\alpha^1_n)^2} & \frac{-\lambda(2)}{(\alpha^2_n)^2(\alpha^1_n)^2} & \ldots & \frac{-\lambda(n-1)}{(\alpha^{n-1}_n)^2(\alpha^1_n)^2} \\
  \end{array}
\right)=
$$
$$
=(-1)^{(n-1)(n-2)}\frac{\lambda(2)\cdots\lambda(n-1)}{(\alpha^1_1)^2\cdots (\alpha^1_{k-1})^2(\alpha^1_{k+1})^2\cdots (\alpha^1_n)^2}
\det \left(
  \begin{array}{cccc}
    1 & \frac{1}{(\alpha^2_1)^2} & \ldots & \frac{1}{(\alpha^{n-1}_1)^2} \\
    \vdots & \vdots & \vdots & \vdots \\
    1 & \frac{1}{(\alpha^2_{k-1})^2} & \ldots & \frac{1}{(\alpha^{n-1}_{k-1})^2} \\
    1 & \frac{1}{(\alpha^2_{k+1})^2} & \ldots & \frac{1}{(\alpha^{n-1}_{k+1})^2} \\
    \vdots & \vdots & \vdots & \vdots \\
    1 & \frac{1}{(\alpha^2_{n})^2} & \ldots & \frac{1}{(\alpha^{n-1}_n)^2} \\
  \end{array} \right) .
$$
Denote by $\nu(l)$ ($l\geq 2$) the quantity $\nu(l)=(\alpha^1_{1})^2-(\alpha^1_{l})^2=\ldots =(\alpha^{n-1}_{1})^2-(\alpha^{n-1}_{l})^2$ then we have
$$
\frac{1}{(\alpha^j_{l})^2}-\frac{1}{(\alpha^j_{1})^2}=\frac{\nu(l)}{(\alpha^j_{l})^2(\alpha^j_{1})^2}
$$
for all $j=1,\ldots ,n-1$.
Then we get
$$
\det B_k=\frac{\lambda(2)\cdots\lambda(n-1)}{(\alpha^1_1)^2\cdots (\alpha^1_{k-1})^2(\alpha^1_{k+1})^2\cdots (\alpha^1_n)^2}
\det \left(
  \begin{array}{cccc}
    1 & \frac{1}{(\alpha^2_1)^2} & \ldots & \frac{1}{(\alpha^{n-1}_1)^2} \\
    \vdots & \vdots & \vdots & \vdots \\
    0 & \frac{\nu(k-1)}{(\alpha^2_{k-1})^2(\alpha^2_{1})^2} & \ldots & \frac{\nu(k-1)}{(\alpha^{n-1}_{k-1})^2(\alpha^{n-1}_{1})^2} \\
    0 & \frac{\nu(k+1)}{(\alpha^2_{k+1})^2(\alpha^2_{1})^2} & \ldots & \frac{\nu(k+1)}{(\alpha^{n-1}_{k+1})^2(\alpha^{n-1}_{1})^2} \\
    \vdots & \vdots & \vdots & \vdots \\
    0 & \frac{\nu(n)}{(\alpha^2_{n})^2(\alpha^2_{1})^2} & \ldots & \frac{\nu(n)}{(\alpha^{n-1}_{n})^2(\alpha^{n-1}_{1})^2} \\
  \end{array} \right)=
$$
$$
=\frac{\lambda(2)\cdots\lambda(n-1)\nu(2)\ldots \nu(k-1)\nu(k+1)\ldots \nu(n)}{\left((\alpha^1_1)^2\cdots (\alpha^1_{k-1})^2(\alpha^1_{k+1})^2\cdots (\alpha^1_n)^2\right)\left((\alpha^2_1)^2\cdots (\alpha^{n-1}_1)^2\right)}\det \left(
  \begin{array}{cccc}
    1 & 1 & \ldots & 1 \\
    0 & \frac{1}{(\alpha^2_{2})^2} & \ldots & \frac{1}{(\alpha^{n-1}_{2})^2} \\
    \vdots & \vdots & \vdots & \vdots \\
    0 & \frac{1}{(\alpha^2_{k-1})^2} & \ldots & \frac{1}{(\alpha^{n-1}_{k-1})^2} \\
    0 & \frac{1}{(\alpha^2_{k+1})^2} & \ldots & \frac{1}{(\alpha^{n-1}_{k+1})^2} \\
    \vdots & \vdots & \vdots & \vdots \\
    0 & \frac{1}{(\alpha^2_{n})^2} & \ldots & \frac{1}{(\alpha^{n-1}_{n})^2} \\
  \end{array} \right)=
$$
$$
=\frac{\lambda(2)\cdots\lambda(n-1)\nu(2)\ldots \nu(k-1)\nu(k+1)\ldots \nu(n)}{\left((\alpha^1_1)^2\cdots (\alpha^1_{k-1})^2(\alpha^1_{k+1})^2\cdots (\alpha^1_n)^2\right)\left((\alpha^2_1)^2\cdots (\alpha^{n-1}_1)^2\right)} \det \left(
  \begin{array}{ccc}
     \frac{1}{(\alpha^2_{2})^2} & \ldots & \frac{1}{(\alpha^{n-1}_{2})^2} \\
     \vdots & \vdots & \vdots \\
     \frac{1}{(\alpha^2_{k-1})^2} & \ldots & \frac{1}{(\alpha^{n-1}_{k-1})^2} \\
     \frac{1}{(\alpha^2_{k+1})^2} & \ldots & \frac{1}{(\alpha^{n-1}_{k+1})^2} \\
     \vdots & \vdots & \vdots \\
     \frac{1}{(\alpha^2_{n})^2} & \ldots & \frac{1}{(\alpha^{n-1}_{n})^2} \\
  \end{array} \right).
$$
Since for $l\geq 3$ we have
$$
\frac{1}{(\alpha^l_{i})^2}-\frac{1}{(\alpha^2_{i})^2}= \frac{(\alpha^2_{i})^2-(\alpha^1_{i})^2+(\alpha^1_{i})^2-(\alpha^l_{i})^2}{(\alpha^l_{i})^2(\alpha^2_{i})^2}= \frac{\lambda(l)-\lambda(2)}{(\alpha^2_{i})^2(\alpha^l_{i})^2},
$$
and
$$
\frac{1}{(\alpha^j_{l})^2}-\frac{1}{(\alpha^j_{2})^2}=\frac{\nu(l)-\nu(2)}{(\alpha^j_{l})^2(\alpha^j_{2})^2}
$$
the next two steps lead to the result
$$
\det B_k=\frac{\prod\limits_{i=2}^{n-1}\lambda(i)\prod\limits_{i=3}^{n-1}(\lambda(i)-\lambda(2))\prod\limits_{\substack{i=2 \\ i\ne k}}^n\nu(i)\prod\limits_{\substack{i=3 \\ i\ne k}}^n(\nu(i)-\nu(2))}{\prod\limits_{\substack{ i=1 \\ i\ne k}}^n(\alpha^1_i)^2\prod\limits_{\substack{i=2 \\i\ne k}}^n (\alpha^2_i)^2\prod\limits_{i=2}^{n-1}(\alpha^i_1)^2\prod\limits_{i=3}^{n-1}(\alpha^i_2)^2}
\det \left(
  \begin{array}{ccc}
     \frac{1}{(\alpha^3_{3})^2} & \ldots & \frac{1}{(\alpha^{n-1}_{3})^2} \\
     \vdots & \vdots & \vdots \\
     \frac{1}{(\alpha^3_{k-1})^2} & \ldots & \frac{1}{(\alpha^{n-1}_{k-1})^2} \\
     \frac{1}{(\alpha^3_{k+1})^2} & \ldots & \frac{1}{(\alpha^{n-1}_{k+1})^2} \\
     \vdots & \vdots & \vdots \\
     \frac{1}{(\alpha^3_{n})^2} & \ldots & \frac{1}{(\alpha^{n-1}_{n})^2} \\
  \end{array} \right).
$$
Using the notation $\lambda(1)=\nu(1)=0$ we get
$$
\det B_k=\frac{\prod\limits_{j=1}^{k-2}\prod\limits_{i=j+1}^{n-1}(\lambda(i)-\lambda(j))\prod\limits_{j=1}^{k-2}\prod\limits_{\substack{i=j+1 \\ i\ne k}}^n(\nu(i)-\nu(j))}{\prod\limits_{\substack{ j=1}}^{k-2}\prod\limits_{\substack{ i=j \\ i\ne k}}^n(\alpha^j_i)^2\prod\limits_{j=1}^{k-2}\prod\limits_{i=j+1}^{n-1}(\alpha^i_j)^2}
\det \left(
  \begin{array}{ccc}
     \frac{1}{(\alpha^{k-1}_{k-1})^2} & \ldots & \frac{1}{(\alpha^{n-1}_{k-1})^2} \\
     \frac{1}{(\alpha^{k-1}_{k+1})^2} & \ldots & \frac{1}{(\alpha^{n-1}_{k+1})^2} \\
     \vdots & \vdots & \vdots \\
     \frac{1}{(\alpha^{k-1}_{n})^2} & \ldots & \frac{1}{(\alpha^{n-1}_{n})^2} \\
  \end{array} \right)=
$$
$$
=\prod\limits_{j=1}^{k-2}\frac{\prod\limits_{i=j+1}^{n-1}(\lambda(i)-\lambda(j))\prod\limits_{\substack{i=j+1 \\ i\ne k}}^n(\nu(i)-\nu(j))}{\prod\limits_{\substack{ i=j \\ i\ne k}}^n(\alpha^j_i)^2\prod\limits_{i=j+1}^{n-1}(\alpha^i_j)^2}
\frac{\prod\limits_{i=k}^{n-1}(\lambda(i)-\lambda(k-1))}{\prod\limits_{\substack{ i=k-1 \\ i\ne k}}^n(\alpha^{k-1}_i)^2}\cdot
$$
$$
\cdot \det \left(
  \begin{array}{cccc}
     1 & \frac{1}{(\alpha^{k}_{k-1})^2} & \ldots & \frac{1}{(\alpha^{n-1}_{k-1})^2} \\
     1 & \frac{1}{(\alpha^{k}_{k+1})^2} & \ldots & \frac{1}{(\alpha^{n-1}_{k+1})^2} \\
     \vdots & \vdots & \vdots & \vdots \\
     1 & \frac{1}{(\alpha^{k}_{n})^2} & \ldots & \frac{1}{(\alpha^{n-1}_{n})^2} \\
  \end{array} \right)=
$$
$$
=\prod\limits_{j=1}^{k-2}\frac{\prod\limits_{i=j+1}^{n-1}(\lambda(i)-\lambda(j))\prod\limits_{\substack{i=j+1 \\ i\ne k}}^n(\nu(i)-\nu(j))}{\prod\limits_{\substack{ i=j \\ i\ne k}}^n(\alpha^j_i)^2\prod\limits_{i=j+1}^{n-1}(\alpha^i_j)^2}
\frac{\prod\limits_{i=k}^{n-1}(\lambda(i)-\lambda(k-1))\prod\limits_{i=k+1}^n(\nu(i)-\nu(k-1))}{\prod\limits_{\substack{ i=k-1 \\ i\ne k}}^n(\alpha^{k-1}_i)^2\prod\limits_{i=k}^{n-1}(\alpha^i_j)^2}\cdot
$$
$$
\cdot \det \left(
  \begin{array}{ccc}
      \frac{1}{(\alpha^{k}_{k+1})^2} & \ldots & \frac{1}{(\alpha^{n-1}_{k+1})^2} \\
      \vdots & \vdots & \vdots \\
      \frac{1}{(\alpha^{k}_{n})^2} & \ldots & \frac{1}{(\alpha^{n-1}_{n})^2} \\
  \end{array} \right)=
$$
$$
=\prod\limits_{j=1}^{k-1}\frac{\prod\limits_{i=j+1}^{n-1}(\lambda(i)-\lambda(j))\prod\limits_{\substack{i=j+1 \\ i\ne k}}^n(\nu(i)-\nu(j))}{\prod\limits_{\substack{ i=j \\ i\ne k}}^n(\alpha^j_i)^2\prod\limits_{i=j+1}^{n-1}(\alpha^i_j)^2}
\det \left(
  \begin{array}{ccc}
      \frac{1}{(\alpha^{k}_{k+1})^2} & \ldots & \frac{1}{(\alpha^{n-1}_{k+1})^2} \\
      \vdots & \vdots & \vdots \\
      \frac{1}{(\alpha^{k}_{n})^2} & \ldots & \frac{1}{(\alpha^{n-1}_{n})^2} \\
  \end{array} \right)=
$$
$$
=\prod\limits_{j=1}^{k-1}\frac{\prod\limits_{i=j+1}^{n-1}(\lambda(i)-\lambda(j))\prod\limits_{\substack{i=j+1 \\ i\ne k}}^n(\nu(i)-\nu(j))}{\prod\limits_{\substack{ i=j \\ i\ne k}}^n(\alpha^j_i)^2\prod\limits_{i=j+1}^{n-1}(\alpha^i_j)^2}\prod\limits_{j=k}^{n-2}\frac{\prod\limits_{i=j+1}^{n-1}(\lambda(i)-\lambda(j)) \prod\limits_{i=j+2}^n(\nu(i)-\nu(j+1))}{\prod\limits_{i=j+1}^n(\alpha^j_i)^2\prod\limits_{i=j+1}^{n-1}(\alpha^i_{j+1})^2}\frac{1}{(\alpha^{n-1}_{n})^2}=
$$
$$
=\frac{\prod\limits_{j=1}^{n-2}\left(\prod\limits_{i=j+1}^{n-1}(\lambda(i)-\lambda(j))\prod\limits_{i=j+1}^n(\nu(i)-\nu(j))\right)} {\prod\limits_{i=1}^{k-1}(\nu(k)-\nu(i))\prod\limits_{i=k+1}^{n}(\nu(i)-\nu(k))}
\frac{\prod\limits_{j=1}^{k-1}(\alpha^j_k)^2\prod\limits_{j=k}^{n-2}(\alpha^j_j)^2\prod\limits_{i=k+1}^{n-1}(\alpha^i_k)^2}
{\prod\limits_{j=1}^{n-2}\left(\prod\limits_{i=j}^n(\alpha^j_i)^2 \prod\limits_{i=j+1}^{n-1}(\alpha^i_j)^2\right)\prod\limits_{j=k}^{n-2}(\alpha^{j+1}_{j+1})^2}\frac{\nu(n)-\nu(n-1)}{(\alpha^{n-1}_{n})^2}=
$$
$$
=\prod\limits_{j=1}^{n-2}\left(\frac{\prod\limits_{i=j+1}^{n-1}(\lambda(i)-\lambda(j))\prod\limits_{i=j+1}^n(\nu(i)-\nu(j))} {\prod\limits_{i=j}^n(\alpha^j_i)^2 \prod\limits_{i=j+1}^{n-1}(\alpha^i_j)^2}\right)\frac{\nu(n)-\nu(n-1)}{(\alpha^{n-1}_{n-1})^2(\alpha^{n-1}_{n})^2} \frac{\prod\limits_{j=1}^{n-1}(\alpha^j_k)^2} {\prod\limits_{i=1}^{k-1}(\nu(k)-\nu(i))\prod\limits_{i=k+1}^{n}(\nu(i)-\nu(k))}.
$$
Denote by $P(n)$ that part of this quantity which is independent from $k$, then
\begin{equation}\label{eq:detBk}
\det B_k=P(n)\frac{\prod\limits_{j=1}^{n-1}(\alpha^j_k)^2} {\prod\limits_{i=1}^{k-1}(\nu(k)-\nu(i))\prod\limits_{i=k+1}^{n}(\nu(i)-\nu(k))}=(-1)^{k-1}P(n)\frac{\prod\limits_{j=1}^{n-1}(\alpha^j_k)^2} {\prod\limits_{\substack{i=1 \\ i\ne k}}^{n}(\nu(i)-\nu(k))}
\end{equation}
with the value:
$$
P(n)=\prod\limits_{j=1}^{n-2}\left(\frac{\prod\limits_{i=j+1}^{n-1}(\lambda(i)-\lambda(j))\prod\limits_{i=j+1}^n(\nu(i)-\nu(j))} {\prod\limits_{i=j}^n(\alpha^j_i)^2 \prod\limits_{i=j+1}^{n-1}(\alpha^i_j)^2}\right)\frac{\nu(n)-\nu(n-1)}{(\alpha^{n-1}_{n-1})^2(\alpha^{n-1}_{n})^2}.
$$
Since $\nu(i)-\nu(k)=(\alpha^1_{1})^2-(\alpha^1_{i})^2-(\alpha^1_{1})^2+(\alpha^1_{k})^2=(\alpha^1_{k})^2-(\alpha^1_{i})^2$,
$N$ is parallel to the vector
$$
\left(\frac{\prod\limits_{j=1}^{n-1}(\alpha^j_1)^2} {\prod\limits_{i=2}^{n}\left((\alpha^1_{1})^2-(\alpha^1_{i})^2\right)}, \ldots, \frac{\prod\limits_{j=1}^{n-1}(\alpha^j_k)^2} {\prod\limits_{\substack{i=1 \\ i\ne k}}^{n}((\alpha^1_{k})^2-(\alpha^1_{i})^2)}, \ldots , \frac{\prod\limits_{j=1}^{n-1}(\alpha^j_n)^2} {\prod\limits_{i=1}^{n-1}((\alpha^1_{n})^2-(\alpha^1_{i})^2)}\right)^T.
$$

This gives the following parametric representation of a common edge: 
\begin{equation}\label{eq:squaredcoord}
x_1=\pm t\sqrt{\frac{\prod\limits_{j=1}^{n-1}(\alpha^j_1)^2}{\prod\limits_{i=2}^{n}\left((\alpha^1_{1})^2-(\alpha^1_{i})^2\right)}}, \ldots , x_n=\pm t\sqrt{\frac{\prod\limits_{j=1}^{n-1}(\alpha^j_n)^2}{\prod\limits_{i=1}^{n-1}((\alpha^1_{n})^2-(\alpha^1_{i})^2)}}.
\end{equation}
The squared distance of a common point from the origin is:
$$
x_1^2+\ldots + x_n^2=t^2\left(\frac{\prod\limits_{j=1}^{n-1}(\alpha^j_1)^2}{\prod\limits_{i=2}^{n}\left((\alpha^1_{1})^2-(\alpha^1_{i})^2\right)}+ \ldots +\frac{\prod\limits_{j=1}^{n-1}(\alpha^j_n)^2}{\prod\limits_{i=1}^{n-1}((\alpha^1_{n})^2-(\alpha^1_{i})^2)}\right).
$$

We prove that the value $t^2$ is itself the squared distance, so we have:
\begin{lemma}\label{lem:identity}
Let $n\geq 2$ and assume that $(\alpha^k_i)^2=(\alpha^1_i)^2+\lambda(k)$ holds for all $i$. Then we have the identity:
\begin{equation}\label{eq:identity}
\left(\frac{\prod\limits_{j=1}^{n-1}(\alpha^j_1)^2}{\prod\limits_{i=2}^{n}\left((\alpha^1_{1})^2-(\alpha^1_{i})^2\right)}+ \ldots +\frac{\prod\limits_{j=1}^{n-1}(\alpha^j_n)^2}{\prod\limits_{i=1}^{n-1}((\alpha^1_{n})^2-(\alpha^1_{i})^2)}\right)=1
\end{equation}
\end{lemma}
\begin{proof}
For $n=2$ we have
$$
\frac{(\alpha^1_1)^2}{\left((\alpha^1_{1})^2-(\alpha^1_{2})^2\right)}+\frac{(\alpha^1_2)^2}{\left((\alpha^1_{2})^2-(\alpha^1_{1})^2\right)}= \frac{(\alpha^1_1)^2-(\alpha^1_2)^2}{(\alpha^1_{1})^2-(\alpha^1_{2})^2}=1.
$$
Similarly, in the three-dimensional case we have
$$
\left(\frac{(\alpha^1_1)^2(\alpha^2_1)^2}{\left((\alpha^1_{1})^2-(\alpha^1_{2})^2\right)\left((\alpha^1_{1})^2- (\alpha^1_{3})^2\right)}+\frac{(\alpha^1_2)^2(\alpha^2_2)^2}{\left((\alpha^1_{2})^2-(\alpha^1_{1})^2\right)\left((\alpha^1_{2})^2- (\alpha^1_{3})^2\right)}+\right.
$$
$$
\left.+\frac{(\alpha^1_3)^2(\alpha^2_3)^2}{\left((\alpha^1_{3})^2-(\alpha^1_{1})^2\right)\left((\alpha^1_{3})^2- (\alpha^1_{2})^2\right)}\right)=
$$
$$
=\left(\frac{(\alpha^1_1)^2(\alpha^2_1)^2\left((\alpha^1_{2})^2- (\alpha^1_{3})^2\right)-(\alpha^1_2)^2(\alpha^2_2)^2\left((\alpha^1_{1})^2- (\alpha^1_{3})^2\right)+(\alpha^1_3)^2(\alpha^2_3)^2\left((\alpha^1_{1})^2-(\alpha^1_{2})^2\right)}{\left((\alpha^1_{1})^2-(\alpha^1_{2})^2\right)\left((\alpha^1_{1})^2- (\alpha^1_{3})^2\right)\left((\alpha^1_{2})^2- (\alpha^1_{3})^2\right)}\right),
$$
and using the connection $(\alpha^2_i)^2=(\alpha^1_i)^2+\lambda(2)$ we get the required equality,
$$
\left(\frac{(\alpha^1_1)^4\left((\alpha^1_{2})^2- (\alpha^1_{3})^2\right)-(\alpha^1_2)^4\left((\alpha^1_{1})^2- (\alpha^1_{3})^2\right)+(\alpha^1_3)^4\left((\alpha^1_{1})^2-(\alpha^1_{2})^2\right)}{\left((\alpha^1_{1})^2- (\alpha^1_{2})^2\right)\left((\alpha^1_{1})^2- (\alpha^1_{3})^2\right)\left((\alpha^1_{2})^2- (\alpha^1_{3})^2\right)}\right)=1.
$$
Hence we can prove by induction.
Let us assume that the identity holds for all $k\leq n-1$ and consider the left hand as a function of $\lambda(n-1)$.
$$
f(\lambda(n-1)):=\left(\frac{\prod\limits_{j=1}^{n-1}(\alpha^j_1)^2}{\prod\limits_{i=2}^{n}\left((\alpha^1_{1})^2-(\alpha^1_{i})^2\right)}+ \ldots +\frac{\prod\limits_{j=1}^{n-1}(\alpha^j_n)^2}{\prod\limits_{i=1}^{n-1}((\alpha^1_{n})^2-(\alpha^1_{i})^2)}\right)=
$$
$$
 =\frac{\prod\limits_{j=1}^{n-2}(\alpha^j_1)^2}{\prod\limits_{i=2}^{n-1}\left((\alpha^1_{1})^2-(\alpha^1_{i})^2\right)} \frac{(\alpha^{1}_1)^2+\lambda(n-1)}{(\alpha^1_{1})^2-(\alpha^1_{n})^2}+ \ldots +\frac{\prod\limits_{j=1}^{n-2}(\alpha^j_{n-1})^2}{\prod\limits_{i=1}^{n-2}((\alpha^1_{n-1})^2-(\alpha^1_{i})^2)} \frac{(\alpha^{1}_{n-1})^2+\lambda(n-1)}{(\alpha^1_{n-1})^2-(\alpha^1_{n})^2}+
$$
$$
 +\frac{\prod\limits_{j=1}^{n-2}(\alpha^j_n)^2}{\prod\limits_{i=1}^{n-2}((\alpha^1_{n})^2-(\alpha^1_{i})^2)} \frac{(\alpha^1_n)^2+\lambda(n-1)}{(\alpha^1_{n})^2-(\alpha^1_{n-1})^2}.
$$
Clearly it is linear in its variable and $f(-(\alpha^1_{n})^2)=1$ by the assumption of the induction. On the other hand if we substitute $\lambda(n-1):=- (\alpha^1_{n-1})^2$ we get a sum (with $(n-1)$ terms) for which the inductive assumption can be applied again, implying that
$$
f(-(\alpha^1_{n-1})^2)=\frac{\prod\limits_{j=1}^{n-2}(\alpha^j_1)^2}{\prod\limits_{\substack{i=2 \\ i\ne n-1}}^{n}\left((\alpha^1_{1})^2-(\alpha^1_{i})^2\right)} + \ldots + \frac{\prod\limits_{j=1}^{n-2}(\alpha^j_n)^2}{\prod\limits_{i=1}^{n-2}((\alpha^1_{n})^2-(\alpha^1_{i})^2)}=1.
$$
Hence $f(\lambda(n-1))=1$ for all values of $\lambda(n-1)$ as we stated.
\end{proof}
From Lemma \ref{lem:identity} we get that the squared length of the segment of the common edges with endpoint $x$ is
\begin{equation}\label{eq:squareddistance}
x_1^2+\ldots + x_n^2=t^2,
\end{equation}
where $t$ is the common parameter in the system of equations in (\ref{eq:squaredcoord}).

\subsection{The tangent cones}

Given a quadric and a point $P$ in its exterior. The union of the tangent lines of the quadric through $P$ is the \emph{tangent cone} to the quadric with \emph{vertex (apex)} $P$. The tangent cone intersects the quadric in the point of the polar hyperplane of $P$ with respect to the quadric, hence this intersection is a quadric of dimension $n-1$. Thus the tangent cone is a surface of second order with centre $P$ and it has $n$ pairwise perpendicular axes.

\begin{statement}\label{st:onthetangentcone}
The axes of the tangent cone are the normals of the confocals of the quadric through the vertex of the cone.
\end{statement}

\begin{proof}
Consider the tangent hyperplane one of these $n$ hypersurfaces which pass through the point $P$. The pole of that plane with regard to the original quadric lies on the polar hyperplane of $P$. We also know that it lies also on the normal
to the examined confocal. It is therefore the point where the normal meets the hyperplane of contact of the cone. The $n$ normals meet the hyperplane of contact in $n$ points, such that each is the pole of the $n-2$-dimensional flat joining the other $n-1$ points with respect to the section quadric in the polar hyperplane of $P$. Hence the $n$ normals form a complete system of conjugate diameters of the cone, and since they are mutually at right angle they are the axes of the cone.
\end{proof}

\subsection{Some properties of the focal quadrics}

The focal quadric $C_k$ (of dimension $n-1$) are the common limits of the distinct types of the confocals. They are lies on the respective principal hyperplanes of the given ellipsoid. Hence $C_k\cap C_l$ for $k\ne l$ is in the $n-2$-dimensional subspace $x_k=x_l=0$. For the points $x$ of this intersection hold the equalities
$$
1=\sum\limits_{\substack{i=1 \\ i\ne k,l}}^n\frac{x_i^2}{a_i^2-a_k^2}\, \mbox{ and } \, 1=\sum\limits_{\substack{i=1 \\ i\ne k,l}}^n\frac{x_i^2}{a_i^2-a_l^2}
$$
showing that they are the common points of two $(n-2)$-dimensional quadrics which are confocals to the ellipsoid $\sum_{i\ne k,l}{x_i^2}/{a_i^2}=1$. Since the points of the intersection hold the equality
$$
\sum\limits_{\substack{i=1 \\ i\ne k,l}}^n\frac{x_i^2}{a_i^2-a_k^2}=\sum\limits_{\substack{i=1 \\ i\ne k,l}}^n\frac{x_i^2}{a_i^2-a_l^2}
$$
we have that
$$
0=\sum\limits_{\substack{i=1 \\ i\ne k,l}}^n\frac{x_i^2(a_k^2-a_l^2)}{(a_i^2-a_k^2)(a_i^2-a_l^2)}=(a_k^2-a_l^2)\sum\limits_{\substack{i=1 \\ i\ne k,l}}^n\frac{x_i^2}{(a_i^2-a_k^2)(a_i^2-a_l^2)}
$$
showing that they are also points of that $(n-2)$-dimensional cone which apex is at the origin and has signed and squared semi-axes $(a_i^2-a_k^2)(a_i^2-a_l^2)$.

\begin{statement}\label{st:axesoffocalcones}
Let $x'$ be any point of the space and denote by $C_k(x')$ the cones with apex $x'$ and generators through the focal quadric $C_k$. Then there are an orthogonal system of lines $l_1,\ldots, l_n$ through $x'$ which elements are the common axes of the focal cones. The magnitudes of the axes corresponding to the same line are dependent from $k$, more precisely the signed and squared lengths of the semi-axes of $C_k(x')$ are $\{\left(a^i_k\right)^2 \,:\, i=1,\ldots,n\}$.
\end{statement}

\begin{proof}
As we saw the focal quadric $C_n$ is the limit of the confocal quadrics if $\lambda$ tends to $a_n$. If $x'$ is any point of the space the sequence of the tangent cones with apex $x'$ of those confocals for which $\lambda $ tends to $a_n$ is tends to the cone $C_n(x')$ with apex $x'$ and through $C_n$. Since in this sequence all cones have the same axes in direction and also in length, this system is the axes of the cone envelops the focal quadric $C_n$. Observe that this is true also any other cone $C_k(x')$ with apex $x'$ and base quadric $C_k$. The directions of the systems of axes are the same for all $k$ but the lengths of the corresponding axes are distinct. Observe that if the $n$ normals be made the axes of coordinates, the equation of the cone must take the form $\sum_{i=1}^n A_ix_i^2=0$. To verify it and determine the concrete coefficients we consider the equation of the tangent cone with apex $x'$ to the original ellipsoid with respect to the original orthogonal basis. It is on the form\footnote{This formula can be calculated from the intersection of the ellipsoid with a line determined by two its points. Substitute the point $\lambda x'+\nu x''$ to the equality of the ellipsoid, then we have the quadratic equation $\lambda^2 U'+2\lambda\nu P +\nu^2 U''=0$, where $U',U'',P$ are the corresponding substitutional values of the quadratic forms in a zero reduced form. If the point $x'$ is the given one and the other $x''$ is on the ellipsoid then one of the solutions corresponds to $\lambda=0$. In order that the other root also corresponds to $\lambda=0$ (U''=0) we must have $P=0$ which is the equation of the polar hyperplane of $x'$ with respect to the ellipsoid. If $x'$ is not on the ellipsoid and the line touch the ellipsoid the quadratic must have equal roots and the coordinates of the two points must be connected with the equality $U'U''=P^2$. Hence the points $x$ of the touching cone satisfies the equation (\ref{eq:tangentcone}).}
\begin{equation}\label{eq:tangentcone}
\left({x'}^TA^{-2}x'-1\right)\left(x^TA^{-2}x-1\right)=\left({x'}^TA^{-2}x-1\right)^2
\end{equation}
Translate it to the new origin $x'$ arises the equation
$$
\left({x'}^TA^{-2}x'-1\right)\left(x^TA^{-2}x\right)=\left({x'}^TA^{-2}x\right)^2
$$
and the orthogonal transformation at this point $x'$ (after a long calculation\footnote{In dimension three this calculation can be found in the book \cite{salmon} in paragraphs 171, 172, 173.}) leads to the final form:
\begin{equation}\label{eq:canformofthetangentcone}
\sum\limits_{i=1}^n\frac{x_i^2}{\left(a^i_1\right)^2-(a_1)^2}=0,
\end{equation}
where the primary axis in the determination of the confocals through $x'$ is taken to the $k$-th axis of the original ellipsoid.
As a particular case may be found the equation (with respect to this new system) of the focal cone $C_k(x')$. In fact, for the square of the primary axis we have to substitute $(a_1)^2-a_k^2=(a^1_1)^2-(a^1_k)^2=\ldots =(a^n_1)^2-(a^n_k)^2$ giving the canonical form
\begin{equation}\label{eq:focalconic}
C_k(x'): \quad \sum\limits_{i=1}^n\frac{x_i^2}{\left(a^i_k\right)^2}=0 \, \mbox{ for } \, k=2,\ldots ,n.
\end{equation}
as we stated.
\end{proof}

Summarizing the calculations (\ref{eq:squaredcoord}), (\ref{eq:squareddistance}) and the result of Statement \ref{st:axesoffocalcones} we have the theorem:
\begin{theorem}\label{thm:commonedges}
The system of focal cones of the ellipsoid
$$
\sum\limits_{i=1}^n\frac{x_i^2}{(a_{i})^2}=1
$$
with common vertex $x'$ with respect to an appropriate coordinate system with origin $x'$ can be given by the equalities
$$
\sum\limits_{i=1}^n\frac{x_i^2}{(a^i_{k+1})^2}=0 \quad k=1,\ldots, n-1,
$$
where $(a^i_{k+1})^2$ is the signed and squared $k$-th semi axis of the $i$-th confocals through the point $x'$.  The points of the common edges hold the parametric system of equation (\ref{eq:squaredcoord}) with coefficients
$$
\frac{\prod\limits_{j=1}^{n-1}(a^1_{j+1})^2}{\prod\limits_{i=2}^{n}(a^1_{2})^2-(a^i_{2})^2},\quad  \ldots \quad , \frac{\prod\limits_{j=1}^{n-1}(a^n_{j+1})^2}{\prod\limits_{i=1}^{n-1}(a^n_{2})^2-(a^i_{2})^2}.
$$
These are the squares of the direction cosines of the corresponding edges, respectively.
\end{theorem}

\begin{example}\label{ex:commonedge}
In the three dimensional case let $a_1=a$, $a_2=b$ and $a_3=c$ with $0<c<b<a$. Clearly, the equation of the focal conics are
$$
\frac{x^2}{a^2-c^2}+\frac{y^2}{b^2-c^2}=1 \quad \mbox{and} \quad z=0,
$$
and
$$
\frac{x^2}{a^2-b^2}+\frac{z^2}{c^2-b^2}=1 \quad \mbox{and} \quad y=0.
$$
If the respective signed and squared axes of the confocals through $x'$ are $(a^i_j)^2$, for $i=1,2,3$ and $j=1,2,3$, then
$$
\mathrm{sign}(a^i_j)^2=\left(\begin{array}{ccc}
+ & + & + \\
+ & + & - \\
+ & - & -
\end{array}\right)
$$
The equation of the focal conics are
$$
\frac{x^2}{(a^1_{2})^2}+\frac{y^2}{(a^2_{2})^2}+\frac{z^2}{(a^3_{2})^2}=0,
$$
and
$$
\frac{x^2}{(a^1_{3})^2}+\frac{y^2}{(a^2_{3})^2}+\frac{z^2}{(a^3_{3})^2}=0.
$$
The square of the searched coefficients are
$$
\frac{(a^1_{2})^2(a^1_{3})^2}{((a^1_{2})^2-(a^2_{2})^2)((a^1_{2})^2-(a^3_{2})^2))}, \quad \frac{(a^2_{2})^2(a^2_{3})^2}{((a^2_{2})^2-(a^1_{2})^2)((a^2_{2})^2-(a^3_{2})^2))}, \quad
\frac{(a^3_{2})^2(a^3_{3})^2}{((a^3_{2})^2-(a^1_{2})^2)((a^3_{2})^2-(a^2_{2})^2))}.
$$
By confocality these three values are positive and determines the direction vector of four lines through the point $x'$. By Lemma \ref{lem:identity} if
\begin{eqnarray}
\nonumber x &= & t\sqrt{\frac{(a^1_{2})^2(a^1_{3})^2}{((a^1_{2})^2-(a^2_{2})^2)((a^1_{2})^2-(a^3_{2})^2)}} \\
\nonumber y &= & t\sqrt{\frac{(a^2_{2})^2(a^2_{3})^2}{((a^2_{2})^2-(a^1_{2})^2)((a^2_{2})^2-(a^3_{2})^2)}} \\
\nonumber z &= & t\sqrt{\frac{(a^3_{2})^2(a^3_{3})^2}{((a^3_{2})^2-(a^1_{2})^2)((a^3_{2})^2-(a^2_{2})^2)}}
\end{eqnarray}
are the three coordinates of a point of the first common edge then its distance from the point $x'$ is equal to $t^2$. By Statement \ref{st:duality} and equality (\ref{eq:onpj}) we get that the square of the first coordinate of the origin is
$$
(p_1)^2=\frac{\left(a_1^1\right)^2\left(a_2^1\right)^2\left(a_3^1\right)^2} {\left(\left(a_1^1\right)^2-\left(a_1^2\right)^2\right)\left(\left(a_1^1\right)^2-\left(a_1^3\right)^2\right)}=\frac{\left(a_1^1\right)^2\left(a_2^1\right)^2\left(a_3^1\right)^2} {\left(\left(a_2^1\right)^2-\left(a_2^2\right)^2\right)\left(\left(a_2^1\right)^2-\left(a_2^3\right)^2\right)}
$$
implying that this is also the first coordinate of that point $X$ on the common edge which is lying on the plane parallel to the tangent hyperplane of the first confocal at $x'$. Hence the distance of $x'$ and $X$ is equal to $t=a_1^1$.
\end{example}

The calculation of the above example is valid in the $n$-dimensional space, too.

\begin{statement}\label{st:lengthoftheintercepts}
$(n-1)$ cones having a common vertex $x$ envelope the $(n-1)$ focal quadrics (of distinct types). The length of the intercept made on one of their common edges by a hyperplane through the origin parallel to the tangent hyperplane to a confocal through $x$ is equal to the major semi-axis of the given confocal.
\end{statement}

\begin{figure}[ht]
  \centering
  \includegraphics[scale=1]{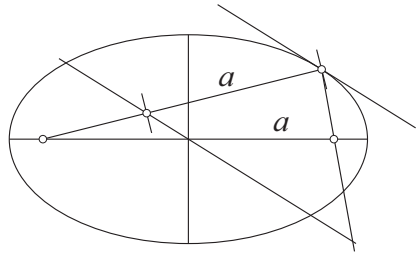}\\
  \caption{Semi-axis-major as an intercept on the focal radii.}\label{fig:intercept}
\end{figure}

\begin{proof}
We use the equations (\ref{eq:squaredcoord}) and (\ref{eq:squareddistance}). Substitute $\alpha_i^{j}:=a^i_{j+1}$ for all $k=1,\ldots ,n-1$ and $i=1,\ldots , n$. Then we get from the duality property (\ref{eq:onpj}) that the first coordinate is equal to
$$
(p^1)^2=\frac{\prod\limits_{i=1}^n\left(a_i^1\right)^2}{\prod\limits_{k=2}^n\left(\left(a_1^1\right)^2-\left(a_1^k\right)^2\right)}=x_1^2,
$$
hence we get from (\ref{eq:squaredcoord}) the equality
$$
\frac{\prod\limits_{i=1}^n\left(a_i^1\right)^2}{\prod\limits_{k=2}^n\left(\left(a_1^1\right)^2-\left(a_1^k\right)^2\right)}=t^2\frac{\prod\limits_{j=1}^{n-1}(a^1_{j+1})^2}{\prod\limits_{i=2}^{n}\left((a^1_{2})^2-(a^i_{2})^2\right)}
$$
implying $t^2=\left(a_1^1\right)^2$ which is by (\ref{eq:squareddistance}) the square of the length of the intercept and also the square of the major semi-axis of the investigated first confocal.
This proves the statement.
\end{proof}
In the plane the above result is well-known. We constructed it in Fig.(\ref{fig:intercept}): the line through the centre of the ellipse and parallel to a tangent to an ellipse cuts off the focal radii portions equal to the semi-axis-major. In the three dimensional case the similar result proved by first Prof. MacCullagh \cite{maccullagh}.

\section{Famous applications in dimension three}

The formulas of the previous sections gives a good frame to prove some known results in the three-space. In this section we examine three such problems.

\subsection{Right cones and focal conics}

The following theorem can be used in mechanics, descriptive geometry and pure geometry of the three-dimensional space, respectively.

\begin{theorem}[See e.g. \cite{stachelbook}]
Let $n=3$. The locus of the apices of right cones through a real focal quadric is the other real focal quadric.
\end{theorem}

\begin{proof}
To find the locus of the vertices of right cones which can envelope a focal quadric we consider the canonical equation (\ref{eq:tangentcone}) of the cone through to $C_k$ with apex $x'$
$$
\sum\limits_{i=1}^n\frac{x_i^2}{\left(a^i_k\right)^2}=0.
$$
It may represent a right cone if $n-1$ from the above coefficients are equals. Since for a fixed $k$ the numbers $\left(a^i_k\right)^2$ are the distinct roots of the polynomial equation (\ref{eq:flambdafunc}),
two of them could be equal to each other if and only if $f((a^j_k)^2)=0$ for an index $j=2, \ldots ,n$ implying that it is the square of two consecutive semi-axes from $a_2, \ldots ,a_n$ which are equal to each other; hence either $(a^j_k)^2=(a_j)^2=(a_{j+1})^2=(a^{j+1}_k)^2$ or $(a^{j-1}_k)^2=(a^j_k)^2=(a_{j-1})^2=(a_{j})^2$ hold. (In both of these cases the index $k$ distinct from the other two ones.) This means that the vertex $x'$ of the cone by definition is a real point of a focal quadric. In the first case, if $x'_j=0$ then it is on $C_j$ and if $x'_{j+1}=0$ then it is on $C_{j+1}$. On the other hand if e.g. we assume that $\left(a^j_k\right)^2=\left(a^{j+1}_k\right)^2$ ($j, j+1\ne k$) then the examined cone with vertex $x'$ and through $C_k$ has canonical form
$$
C_k(x'): \quad \sum\limits_{\substack{i=1 \\ i\ne j,j+1}}^{j-1}\frac{x_i^2}{\left(a^i_k\right)^2}+\frac{x_j^2}{\left(a^j_k\right)^2}+\frac{x_{j+1}^2}{\left(a^{j}_k\right)^2}=0
$$
showing that it has two semi-axes which are equal to each other.

In the three-dimensional case for $k=3$ the index $j$ may be $1$ or $2$. Since $C_1$ has no real point we know that $C_2$ contains the vertex of $C_3(x')$ and the latter is a right cone and vice versa. The equations of the two focal conics are
$$
C_3:\, \left\{\frac{x_1^2}{a_1^2-a_3^2}+\frac{x_2^2}{a_2^2-a_3^2}=1,\, x_3=0\right\} \, \mbox{ and } \, C_2:\, \left\{\frac{x_1^2}{a_1^2-a_2^2}+\frac{x_3^2}{a_3^2-a_2^2}=1,\, x_2=0\right\}
$$
and the canonical equation of the right cone through $C_3$ and with the vertex $x'\in C_2$ is
$$
\frac{x_1^2}{a_1^2}+\frac{x_2^2}{a_1^2}+\frac{x_3^2}{a_3^2}=0.
$$
\end{proof}

\subsection{Chasles's construction on conjugate diameters}

The standard proof of Rytz construction based on the so-called "two-circle figure" (see in Fig. \ref{fig:rytz}) in which we draw the incircle and the circumcircle of the ellipse. The same figure with a little modification enable to get another construction to solve this problem.

\begin{figure}[ht]
  \centering
  \includegraphics[width=15cm]{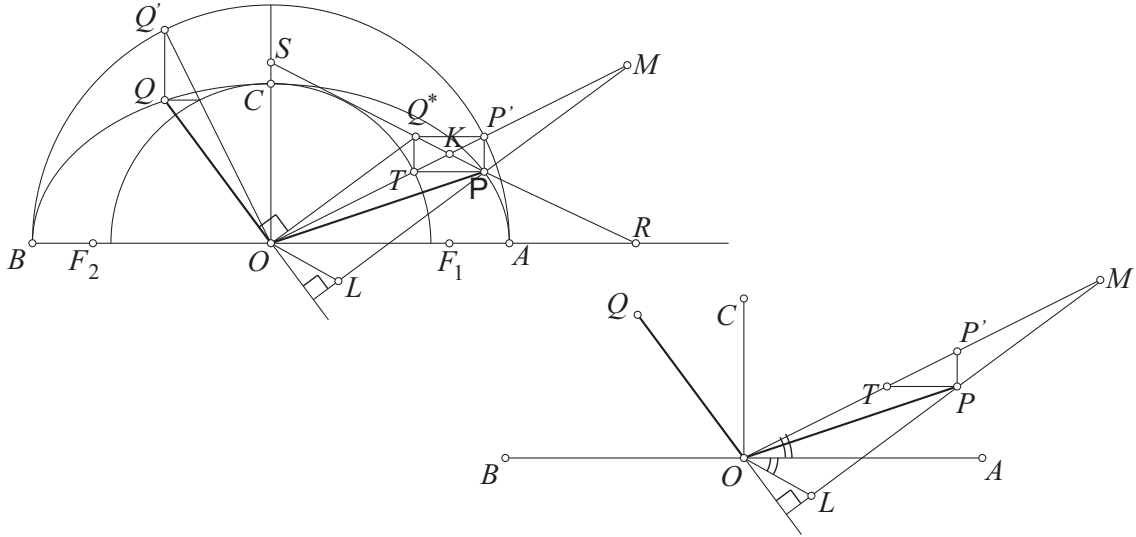}\\
  \caption{Conjugate diameters and axes }\label{fig:rytz}
\end{figure}

In fact, let the line $PM$ be the normal of the ellipse at $P$. This is perpendicular to $OQ$ hence it is parallel to $OQ^\star$. If $M$ is the intersection of this normal with the line $OK$ (as in the standard proof $Q^\star$ is the rotated copy of $OQ$ by $\pi/2$ and $K$ is the middle point of the segment $PQ^\star $), then the triangles $OTQ^\star$ and $MP'P$ are congruent to each other. From this we get that $PM$ is congruent to $OQ^\star$. Let $L$ be the reflected image of $M$ at $P$. Since $PL$ is parallel and equal to $OQ^\star$, $OL$ is also parallel and congruent to $Q^\star P$. Hence the axis $AB$ of the ellipse is the bisector of the angle $LOM\angle$. The other axis of the ellipse is perpendicular to $AB$ at $O$, and if we draw parallels from $P$ to these axes we get the points $T$ and $P'$ as the intersections of the line $OM$ with these parallels, respectively. Clearly, the lengths of $OT$ and $OP'$ are equals to the lengths of the semi-axes, respectively. The construction has the following steps:

\begin{itemize}
\item Draw perpendicular to line $OQ$ from $P$ and determine the points $M$, $L$ on this line by the property $|PM|=|PL|=|OQ|$.
\item Draw the bisectors of the angle $LOM\angle$, these are the axes of the ellipse.
\item Draw parallels to these bisectors from the point $P$ and determine the intersections of these lines with the line $OM$ ($T$, $P'$). The lengths of the semi-axes are the lengths of the segments $OT$, $OP'$, respectively.
\end{itemize}

Observe that $L$ and $M$ determine a tangential pencil of confocal conics in the plane. From these pencils there are two conics an ellipse and a hyperbole through the point $O$. By the bisector property, the tangent lines of these conics at $O$ are the axes of the searching ellipse. If we visualize the union of lines $OM$ and $OL$ as a cone with apex $O$ and through the focal figure of the confocal system of conics determined by the points $M$ and $L$ then the axes of the ellipse are equals the axes of this cone. Hence this construction contains only such concepts whose analogous exist in higher dimensions, too.

Let us follow our investigations in the three-space. Our purpose to give a construction, which determine the axes of an ellipsoid from a complete system of conjugate diameters.
To this consider the three pairwise conjugate semi-diameters $OP$,$OQ$ and $OR$, respectively. The construction of Chasles (see in \cite{chasles} or \cite{salmon}) contains the following steps:
\begin{itemize}
\item Determine the three mutually perpendicular line through $P$, one of its the normal of the ellipsoid (perpendicular to the plane $(OQR)$ at $P$) and the other two are parallel to the axes of the ellipse $\mathcal{E}$ which is the intersection of the plane $(OQR)$ by the given ellipsoid. (By the above construction we can get these lines.)

\item They contains the axes of a confocal system of quadrics dual to the confocal system of the given ellipsoid. By Statement \ref{st:duality} the axes of this systems can be given by the system $\{a_1^1,a_1^2,a_1^3\}$ as respective lengths of the semi-axes. The square of the semi-axes of the focal conics of this system are $\{(a_1^1)^2-(a_1^3)^2,(a_1^2)^2-(a_1^3)^2\}$; $\{(a_1^1)^2-(a_1^2)^2,(a_1^3)^2-(a_1^2)^2\}$ giving an ellipse and a hyperbole, respectively. The square of the axes of $\mathcal{E}$ are $(\rho_2^1)^2=(a_1^1)^2-(a_1^2)^2$ and $(\rho_3^1)^2=(a_1^1)^2-(a_1^3)^2$. These values are known from the construction, hence we can give also the squares of the lengths of the semi-axes of the focal conics.

\item The intersection of the two focal cones with apex $O$ is the union of four common edges, the six planes determining by these edges intersect to each other in three mutual perpendicular lines which are the three axes of these two confocal cones. To prove this statement observe that the two focal cones are confocal quadrics and so they principal axes are common in their directions. The three principal axis intersects a plane of intersection in three points which are form an autopolar triangle of this plane with respect to the two intersection conics by the two focal cones. The four common points of the two conic of intersections determines that quadrangle which diagonal points form the only autopolar triangle with respect to both of the conics. By Statement \ref{st:axesoffocalcones} these are the searched axes of the given ellipsoid, while the planes through $P$ and parallel to the principal planes of the ellipsoid cut off on these four lines parts equal in length to the semi-axes by Statement \ref{st:lengthoftheintercepts}.
\end{itemize}

Observe that this theoretical construction cannot be constructed with ruler and compass in general, because the construction of the common lines of two cones of second order implies the construction of the common points of two conics, which is some cases unconstructible.  More details on this problem can be found in \cite{gho-prok}.

\subsection{Staude's wire model in $3$-space}

At the end of the eighteenth century O. Staude printed the book \cite{staude} with a lot of interesting results. H. D. Thompson write his recension on the book the following: \emph{Possibly the most interesting, although not the most general of the focal properties which Professor Staude deduces, are the following. (The "focal distance" from any point to the focus of the one of a pair of focal conics is defined as the shortest distance on the broken line to a point on that conic, and thence to the adjacent focus). In an ellipsoid, for any point on the surface, the sum of the focal distances to a focus of the focal ellipse and to the opposite focus of the focal hyperbola is a constant. And the normal to the surface at the point bisects the angle of the two lines. This theorem is the basis of the well-known wire model, No. 110 of the Brill collection, which represents the two focal curves of the system by wires rigidly connected, and which has a string fastened so as to represent the sum of the focal distances of the theorem.} This result can be cited without proof in the book of Hilbert and Cohn-Vossen \cite{hilbert-cohnvossen}, too.

\begin{lemma}\label{lem:fastenedstring}
Let $r(u_1,\ldots,u_{n-1})=(x_1(u_1,\ldots,u_{n-1}),\ldots ,x_n(u_1,\ldots,u_{n-1}))$ be a hypersurface of the $n$-dimensional Euclidean space. $F$ and $P$ are two points of the space and $Q$ is a point of the surface for which the length of the broken line $FQP$ is minimal. Then the angles between the normal vector of the surface at its point $Q$ and the respective segments $FQ$ and $QP$ are equal to each other.
\end{lemma}

\begin{proof}
Denote the coordinates of $F$ and $P$ by $(f_1,\ldots f_n)^T$ and $(p_1,\ldots, p_n)^T$, respectively. We also denote by $f=\sum_{i=1}^n(x_i-f_i)^2$ and $g=\sum_{i=1}^n(-x_i+p_i)^2$ the functions $|FP|^2$ and $|PQ|^2$, respectively. Then the partial derivatives of the length function $\sqrt{f}+\sqrt{g}$ of variable $u_1,\ldots, u_n$ have to vanish implying the equalities:
$$
0=\frac{\frac{\partial f}{\partial u_i}}{2|FQ|}+\frac{\frac{\partial g}{\partial u_i}}{2|QP|}=\frac{\sum\limits_{j=1}^{n}(x_j-f_j)\frac{\partial x_j}{\partial u_i}}{|FQ|}-\frac{\sum\limits_{j=1}^{n}(x_j-p_j)\frac{\partial x_j}{\partial u_i}}{|PQ|}=\sum\limits_{j=1}^{n}\left(\frac{(x_j-f_j)}{|FQ|}-\frac{(x_j-p_j)}{|PQ|}\right)\frac{\partial x_j}{\partial u_i}
$$
for all $i=1,\ldots (n-1)$. Hence $FQ/|FQ|-PQ/|PQ|$ orthogonal to $n-1$ independent vectors of the tangent hyperplane at the searched point $Q$ meaning that it is parallel to the normal of the hypersurface at this point $Q$. Since the vectors $FQ/|FQ|$ and $PQ/|PQ|$ are unit vectors this proves the lemma.
\end{proof}
The following result can be found in \cite{stachelbook}, too.
\begin{lemma}\label{lem:hhstring}
Assume that $H_1$ and $H_2$ are two points on the same branch (on the different branches) of the focal hyperbola and $E$ is any point of the focal ellipse. Then the difference $|H_1E|-|EH_2|$ (the sum $|H_1E|+|EH_2|$) of the lengths of the edges of the broken line $H_1EH_2$ is independent from the position of $E$. Similarly if $E_1$ and $E_2$ are two points on the focal ellipse then the difference $|E_1H|-|HE_2|$ of the  broken line $E_1HE_2$ is independent from the position of $H$ in the focal hyperbola.
\end{lemma}

\begin{proof}
Using the notation of Example \ref{ex:commonedge} the focal ellipse contains that points of the space which are of the form
$(\sqrt{a^2-c^2}\cos u, \sqrt{b^2-c^2}\sin u, 0)^T$ for $0\leq u< 2\pi$ and the points of the two branches of the focal hyperbola have the coordinates
$(\pm\sqrt{a^2-b^2}\cosh v, 0, -\sqrt{b^2-c^2}\sinh v)^T$. If $H_i$ are the same branch of the hyperbola then we have
$$
|H_1E|-|EH_2|=\sqrt{(\sqrt{a^2-c^2}\cos u \mp\sqrt{a^2-b^2}\cosh v_1)^2+(b^2-c^2)^2\sin ^2 u+(b^2-c^2)^2\sinh ^2 v_1}-$$
$$
-\sqrt{(\sqrt{a^2-c^2}\cos u \pm\sqrt{a^2-b^2}\cosh v_2)^2+(b^2-c^2)^2\sin ^2 u+(b^2-c^2)^2\sinh ^2 v_2}=
$$
$$
=\sqrt{(a^2-b^2)\cos^2 u\mp 2\sqrt{a^2-c^2}\cos u\sqrt{a^2-b^2}\cosh v_1+(a^2-c^2)\cosh^2 v_1}-
$$
$$
-\sqrt{(a^2-b^2)\cos^2 u\mp 2\sqrt{a^2-c^2}\cos u\sqrt{a^2-b^2}\cosh v_2+(a^2-c^2)\cosh^2 v_2}=
$$
$$
=|\sqrt{a^2-b^2}\cos u\mp \sqrt{a^2-c^2}\cosh v_1|-|\sqrt{a^2-b^2}\cos u\mp \sqrt{a^2-c^2}\cosh v_2|.
$$
Since $\sqrt{a^2-c^2}\cosh v>\sqrt{a^2-b^2}\cos u$ for all values $u$ and $v$ the terms containing the parameter $u$ are vanishing showing the truth of the statement in this case. The other two statements of the lemma can be proved similarly.
\end{proof}

\begin{lemma}\label{lem:roleofcommonedges}
If $P$ is any point of the space and $l$ is a common transversal of the focal conics through $P$ with points of intersection $E\in C_3\cap l$ and $H\in C_2\cap l$ for which $E$ separates $P$ and $H$ (for which $H$ separates $P$ and $E$). Then for any point $H_1$ ($E_1$) on the same branch of the focal hyperbola (of the focal ellipse) as $H$ is, the minimal length of the broken line $PFH_1$ ($PFE_1$) from $P$ to $H_1$ ($E_1$) through a point $F$ of the focal ellipse (of the same branch of the focal hyperbola as $H$) attend at the point $E$ ($H$).

If $P$ separates the two points of intersection, then we have the inequality $|PE|-|EH_1|\leq |PF|-|FH_1|$.
\end{lemma}

\begin{proof}
$|PE|+|EH|\leq |PF|+|FH|$ for all points of $C_3$. $|HE|-|EH_1|=|HF|-|FH_1|$ implying that
$$
|PE|+|EH_1|=|PE|+|EH|-(|EH|-|EH_1|)\leq |PF|+|FH|-(|FH|-|FH_1|)=|PF|+|FH_1|,
$$
as we stated. The alternative statement in the brackets can be get with the same manner.

If $P$ separates $E$ and $H$ then $|EH|-|PE|\geq |FH|-|FP|$ and again $|HE|-|EH_1|=|HF|-|FH_1|$, hence $|PE|-|EH_1|\leq |PF|-|FH_1|$ as we stated.
\end{proof}

Now we are ready to prove Staude's result on wire construction.

\begin{theorem}\label{thm:staude}
Let $P$ be a point of the ellipsoid $\mathcal{E}$ given by the canonical form
$$
\frac{x^2}{a^2}+\frac{y^2}{b^2}+\frac{z^2}{c^2}=1.
$$
Denote by $F_1=(-\sqrt{a^2-c^2}, 0, 0)^T$ the left focus of the focal hyperbola of the confocal quadrics defined by the given ellipsoid and $G_2=(\sqrt{a^2-b^2},0,0)^T$ is the right focus of the focal ellipse of the same system. Then the sum of the shortest length of the broken line $F_1HPEG_2$ where $H$ is a point of the focal hyperbola and $E$ is a point of the focal ellipse is equal to $2a+\sqrt{a^2-c^2}-\sqrt{a^2-b^2}$ hence it is independent from the position of $P$ on the ellipsoid.
\end{theorem}

\begin{proof}
To prove this statement use first Lemma \ref{lem:roleofcommonedges} to calculate that distances which are needed for us. The common transversals of the focal conics have the parametric representation
\begin{eqnarray}\label{eqn:commonedges}
\nonumber \xi &= & \varepsilon(\xi) t\sqrt{\frac{b^2c^2}{\mu\nu}} \\
 \eta &= & \varepsilon(\eta) t\sqrt{\frac{(b^2-\mu)(\mu-c^2)}{\mu(\nu-\mu)}} \\
\nonumber \zeta &= & \varepsilon(\zeta) t\sqrt{\frac{(\nu-b^2)(\nu-c^2)}{\nu(\nu-\mu)}}
\end{eqnarray}
with respect to the dual coordinate system with origin $P$ and having the normals of the confocals through $P$ as axes. Here $a^1_1=a$, $a^1_2=b$, $a^1_3=c$; $(a^2_1)^2=a^2-\mu$, $(a^2_2)^2=b^2-\mu$, $(a^2_3)^2=c^2-\mu$; and $(a^3_1)^2=a^2-\nu$, $(a^3_2)^2=b^2-\nu$, $(a^3_3)^2=c^2-\nu$ with $c^2<\mu <b^2<\nu<a^2$. One of the signs $\varepsilon(\xi), \varepsilon(\eta), \varepsilon(\zeta)$ can be chosen arbitrarily.  The coordinate transformation which connects the variable $\xi,\eta,\zeta$ to the original variable $x,y,z$ is
\begin{eqnarray}\label{eqn:transformation}
\nonumber x-x'& = & \left(\frac{p_0}{a^2}\xi+\frac{p_\mu}{a^2-\mu}\eta +\frac{p_\nu}{a^2-\nu}\zeta\right)x' \\
 y-y'& = & \left(\frac{p_0}{b^2}\xi+\frac{p_\mu}{b^2-\mu}\eta +\frac{p_\nu}{b^2-\nu}\zeta\right)y' \\
\nonumber z-z'& = & \left(\frac{p_0}{c^2}\xi+\frac{p_\mu}{c^2-\mu}\eta +\frac{p_\nu}{c^2-\nu}\zeta\right)z',
\end{eqnarray}
where the coordinates of the point $P$ with respect to the original ($\{O,x,y,z\}$) system are $x',y'$ and $z'$, respectively and by (\ref{eq:onpj})
\begin{equation}\label{eq:thepvalues}
p_0:=\sqrt{\frac{a^2b^2c^2}{\mu\nu}},\, p_\mu=\sqrt{\frac{(a^2-\mu)(b^2-\mu)(\mu-c^2)}{\mu(\nu-\mu)}},\, p_\nu=\sqrt{\frac{(a^2-\nu)(\nu-b^2)(\nu-c^2)}{\nu(\nu-\mu)}}.
\end{equation}
Hence the common edges with respect to the system $\{O,x,y,z\}$ have the following form:
\begin{eqnarray}\label{eqn:transformedcommedge}
\nonumber x-x'& = & t\left(\varepsilon(\xi)\frac{b^2c^2}{a\mu\nu}-\varepsilon(\eta) \frac{(b^2-\mu)(c^2-\mu)}{\sqrt{a^2-\mu}\mu(\nu-\mu)}+\varepsilon(\zeta) \frac{(b^2-\nu)(c^2-\nu)}{\sqrt{a^2-\nu}\nu(\nu-\mu)}\right)x' \\
 y-y'& = & t\left(\varepsilon(\xi)\frac{ac^2}{\mu\nu}-\varepsilon(\eta) \frac{\sqrt{a^2-\mu}(c^2-\mu)}{\mu(\nu-\mu)} +\varepsilon(\zeta) \frac{\sqrt{a^2-\nu}(c^2-\nu)}{\nu(\nu-\mu)}\right)y' \\
\nonumber z-z'& = & t\left(\varepsilon(\xi)\frac{ab^2}{\mu\nu}-\varepsilon(\eta) \frac{\sqrt{a^2-\mu}(b^2-\mu)}{\mu(\nu-\mu)}+\varepsilon(\zeta) \frac{\sqrt{a^2-\nu}(b^2-\nu)}{\nu(\nu-\mu)}\right)z'.
\end{eqnarray}
The focal ellipse lying on the plane $z=0$ implying that the corresponding possible parameters are
$$
t=\frac{-1}{\left(\varepsilon(\xi)\frac{ab^2}{\mu\nu}-\varepsilon(\eta) \frac{\sqrt{a^2-\mu}(b^2-\mu)}{\mu(\nu-\mu)} +\varepsilon(\zeta) \frac{\sqrt{a^2-\nu}(b^2-\nu)}{\nu(\nu-\mu)}\right)}=
$$
$$
=\frac{\mu\nu(\nu-\mu)}{-\varepsilon(\xi) ab^2(\nu-\mu)+\varepsilon(\eta) \sqrt{a^2-\mu}(b^2-\mu)\nu -\varepsilon(\zeta)\sqrt{a^2-\nu}(b^2-\nu)\mu},
$$
which square gives the squared distance of $P$ to the corresponding point $E$ of the transversal by equation (\ref{eq:squareddistance}).
From this we get
$$
 x=\frac{(a^2-c^2)\left(-\varepsilon(\xi)\frac{b^2
(\nu-\mu)}{a}+\varepsilon(\eta) \frac{\nu(b^2-\mu)}{\sqrt{a^2-\mu}}-\varepsilon(\zeta) \frac{\mu(b^2-\nu)}{\sqrt{a^2-\nu}}\right)}{-\varepsilon(\xi) ab^2(\nu-\mu)+\varepsilon(\eta) \sqrt{a^2-\mu}(b^2-\mu)\nu -\varepsilon(\zeta)\sqrt{a^2-\nu}(b^2-\nu)\mu}x'=
$$
$$
=\frac{(a^2-c^2)\left(-\varepsilon(\xi)\frac{b^2
(\nu-\mu)}{a}+\varepsilon(\eta) \frac{\nu(b^2-\mu)}{\sqrt{a^2-\mu}}-\varepsilon(\zeta) \frac{\mu(b^2-\nu)}{\sqrt{a^2-\nu}}\right)}{-\varepsilon(\xi) ab^2(\nu-\mu)+\varepsilon(\eta) \sqrt{a^2-\mu}(b^2-\mu)\nu -\varepsilon(\zeta)\sqrt{a^2-\nu}(b^2-\nu)\mu}\varepsilon(x')\frac{a\sqrt{a^2-\mu}\sqrt{a^2-\nu}}{\sqrt{a^2-b^2}\sqrt{a^2-c^2}}=
$$
$$
=\varepsilon(x')\frac{\sqrt{a^2-c^2}}{\sqrt{a^2-b^2}}\frac{\left( -\varepsilon(\xi) b^2
(\nu-\mu)\sqrt{a^2-\mu}\sqrt{a^2-\nu}+\varepsilon(\eta) \nu(b^2-\mu)a\sqrt{a^2-\nu} -\varepsilon(\zeta)\mu(b^2-\nu)a\sqrt{a^2-\mu}\right)}{\left(-\varepsilon(\xi) ab^2(\nu-\mu)+\varepsilon(\eta) \sqrt{a^2-\mu}(b^2-\mu)\nu -\varepsilon(\zeta)\sqrt{a^2-\nu}(b^2-\nu)\mu\right)},
$$
where $\varepsilon(x')$ is the sign of $x'$ and we used here equation (\ref{eq:coordwithconfoc}), too. For brevity denote by the long fraction by $B/A$ meaning that
$$
x:=\varepsilon(x')\frac{\sqrt{a^2-c^2}}{\sqrt{a^2-b^2}}\frac{B}{A}.
$$
Applying in the plane $z=0$ the equation (\ref{eq:coordwithconfoc}) for the focal ellipse we get that there is a constant $e$ for which $b^2\leq e^2\leq a^2$ and
$$
x=\varepsilon((C_3)_x)\frac{\sqrt{a^2-c^2}\sqrt{a^2-e^2}}{\sqrt{a^2-b^2}} \mbox{ and } y=\varepsilon((C_3)_y)\frac{\sqrt{b^2-c^2}\sqrt{e^2-b^2}}{\sqrt{a^2-b^2}}
$$
where $\varepsilon((C_3)_x)$ is the sign of the coordinate $x$. This implies the equality:
$$
\varepsilon(x')\varepsilon((C_3)_x)\sqrt{a^2-e^2}=\frac{B}{A}.
$$
On the other hand we have the identity
$$
\varepsilon(\xi)\varepsilon(\eta)\varepsilon(\zeta)B+(\varepsilon(\xi) a+\varepsilon(\eta)\sqrt{a^2-\mu}+\varepsilon(\zeta)\sqrt{a^2-\nu})A=
$$
$$
=-\varepsilon(\eta)\varepsilon(\zeta)b^2(\nu-\mu)\sqrt{a^2-\mu}\sqrt{a^2-\nu}+\varepsilon(\xi)\varepsilon(\zeta) \nu(b^2-\mu)a\sqrt{a^2-\nu}-\varepsilon(\xi)\varepsilon(\eta) \mu(b^2-\nu)a\sqrt{a^2-\mu}+
$$
$$
+(\varepsilon(\xi) a+\varepsilon(\eta)\sqrt{a^2-\mu}+\varepsilon(\zeta)\sqrt{a^2-\nu})(-\varepsilon(\xi) ab^2(\nu-\mu)+\varepsilon(\eta) \sqrt{a^2-\mu}(b^2-\mu)\nu -\varepsilon(\zeta)\sqrt{a^2-\nu}(b^2-\nu)\mu)=
$$
$$
=\varepsilon(\eta)\varepsilon(\zeta)\left(-b^2(\nu-\mu)-(b^2-\nu)\mu+(b^2-\mu)\nu\right)\sqrt{a^2-\mu}\sqrt{a^2-\nu}+
$$
$$
+\varepsilon(\xi)\varepsilon(\zeta) \left(\nu(b^2-\mu)-(b^2-\nu)\mu-b^2(\nu-\mu)\right)a\sqrt{a^2-\nu} +
$$
$$
+\varepsilon(\xi)\varepsilon(\eta) \left(-\mu(b^2-\nu) - b^2(\nu-\mu)+(b^2-\mu)\nu\right)a\sqrt{a^2-\mu}+
$$
$$
-a^2b^2(\nu-\mu)+ (a^2-\mu)(b^2-\mu)\nu -(a^2-\nu)(b^2-\nu)\mu=\mu^2\nu-\nu^2\mu=-\mu\nu(\nu-\mu),
$$
implying the result:
\begin{equation}\label{eq:onthedistanceellipse}
t=\frac{\mu\nu(\nu-\mu)}{A}=-\varepsilon(\xi)a-\varepsilon(\eta)\sqrt{a^2-\mu}-\varepsilon(\zeta)\sqrt{a^2-\nu}- \varepsilon(\xi)\varepsilon(\eta)\varepsilon(\zeta)\varepsilon(x')\varepsilon((C_3)_x)\sqrt{a^2-e^2}.
\end{equation}

Similarly the focal hyperbola is lying on the plane $y=0$. Hence we have in this case that
$$
\tau=\frac{\mu\nu(\nu-\mu)}{-\varepsilon(\xi)ac^2(\nu-\mu)+\varepsilon(\eta) \sqrt{a^2-\mu}(c^2-\mu)\nu -\varepsilon(\zeta)\sqrt{a^2-\nu}(c^2-\nu)\mu},
$$
and also
$$
\bar{x}=\varepsilon(x')\frac{\sqrt{a^2-b^2}}{\sqrt{a^2-c^2}}\cdot
$$
$$
\cdot\frac{-\varepsilon(\xi)c^2
(\nu-\mu)\sqrt{a^2-\mu}\sqrt{a^2-\nu}+\varepsilon(\eta) \nu(c^2-\mu)a\sqrt{a^2-\nu}-\varepsilon(\zeta) \mu(c^2-\nu)a\sqrt{a^2-\mu}}{-\varepsilon(\xi)ac^2(\nu-\mu)+\varepsilon(\eta)\sqrt{a^2-\mu}(c^2-\mu)\nu -\varepsilon(\zeta)\sqrt{a^2-\nu}(c^2-\nu)\mu}
=
$$
$$
=\varepsilon(x')\frac{\sqrt{a^2-b^2}}{\sqrt{a^2-c^2}}\frac{D}{C}
$$
For the coordinates of the focal hyperbola by (\ref{eq:coordwithconfoc}) hold that
$$
\bar{x}=\varepsilon((C_2)_x)\frac{\sqrt{a^2-b^2}\sqrt{a^2-f^2}}{\sqrt{a^2-c^2}},\, y=0 \mbox{ and } z=\varepsilon((C_2)_z)\frac{\sqrt{b^2-c^2}\sqrt{f^2-c^2}}{\sqrt{a^2-c^2}},
$$
hence we have
$$
\varepsilon(x')\varepsilon((C_2)_{\bar{x}})\sqrt{a^2-f^2}=\frac{D}{C}.
$$
 Clearly the  identity
$$
\varepsilon(\xi)\varepsilon(\eta)\varepsilon(\zeta)D+(\varepsilon(\xi)a+\varepsilon(\eta)\sqrt{a^2-\mu}+\varepsilon(\zeta)\sqrt{a^2-\nu})C=-\mu\nu(\nu-\mu)
$$
is also hold, and thus
\begin{equation}\label{eq:onthedistancehyp}
\tau=\frac{\mu\nu(\nu-\mu)}{C}=-\varepsilon(\xi)a-\varepsilon(\eta)\sqrt{a^2-\mu}-\varepsilon(\zeta)\sqrt{a^2-\nu}- \varepsilon(\xi)\varepsilon(\eta)\varepsilon(\zeta)\varepsilon(x')\varepsilon((C_2)_x)\sqrt{a^2-f^2}.
\end{equation}
For a given position of the point $P$ choose that four combinations for the eight possible ones that the four intersection points on the focal ellipse correspond to positive parameter values, respectively. It is easy to see\footnote{We have to prove that the equality $ab^2(\nu-\mu)>\sqrt{a^2-\mu}(b^2-\mu)\nu+\sqrt{a^2-\nu}(b^2-\nu)\mu$ fulfill if $c^2<\mu<b^2<\nu<a^2$ hold for the corresponding parameters.}  that in this case we have to choose the following combination of signs:
$$
\begin{array}{|c|c|c|c|}
\hline
&\varepsilon(\xi) & \varepsilon(\eta) & \varepsilon(\zeta) \\
\hline
\hline
r_1 & -1 & 1 & 1 \\
\hline
r_2 &-1 & 1 & -1 \\
\hline
r_3 & -1 & -1 & -1 \\
\hline
r_4 &-1 & -1 & 1 \\
\hline
\end{array},
$$
The corresponding common transversals we call the focal radiuses $r_i$ of the point $P$ to the focal conics.

In this case, we can compare the positions of the points $E$ and $H$ on a focal radius by the parameter values $t$ and $\tau$. If $\tau > t>0$ then the point $E$ separates the point $P$ and $H$ and if $t>\tau$ then either $H$ separates $P$ and $E$ or $P$ separates $E$ and $H$, respectively.
Consider the difference
$$
\frac{1}{\tau}-\frac{1}{t}=-\frac{b^2-c^2}{\mu\nu(\nu-\mu)}\left(-\varepsilon(\xi)a(\nu-\mu)+\varepsilon(\eta) \sqrt{a^2-\mu}\nu -\varepsilon(\zeta)\sqrt{a^2-\nu}\mu\right)=
$$
$$
=\frac{b^2-c^2}{\mu\nu(\nu-\mu)}(\varepsilon(\eta)\sqrt{a^2-\mu}-\varepsilon(\zeta)\sqrt{a^2-\nu}) (\varepsilon(\zeta)\sqrt{a^2-\nu}-\varepsilon(\xi)a)(\varepsilon(\xi)a-\varepsilon(\eta)\sqrt{a^2-\mu}).
$$
This value is positive for $r_3$ and $r_4$ and negative for $r_1$ and $r_2$, respectively. This implies that for $i=1,2$ we have two possibilities. If $\tau> o$ then the focal radius $r_i$ first intersect the focal ellipse and than the focal hyperbola and if $\tau<0$ then $P$ separates the two intersection points to each other. In the case of $i=3,4$ $\tau$ is always positive the two points of intersections are on the same radius and the point of the focal hyperbola separates the point $P$ and the point of the focal ellipse. For $i=3$ the numerator $D$ of $r_3$ is negative and the same value for $r_4$ is positive. Since the denominator $C$ is positive in both of these cases, the corresponding points of the focal hyperbola have opposite half-spaces with respect to the $yz$-plane. Hence in the case of $r_3$ $\varepsilon((C_2)_x)=-\varepsilon(x')$ and in the case of $r_4$ $\varepsilon((C_2)_x)=\varepsilon(x')$, respectively. Since for $r_1$ $D$ is positive and for $r_2$ $D$ is negative then the sign dependent from the relative position of the point $P$ and the points of the focal hyperbola. For $r_1$ with respect to the fact that $\tau$ is positive or negative, $P$ has to the same or the opposite half-space as the point of the focal hyperbola. Conversely, for $r_2$ the point of the focal hyperbola lies in the opposite or same half-space as $P$ with respect to the fact that $tau$ positive or negative, respectively. By formula we have that in the case of $r_1$ $\varepsilon((C_2)_x)=\varepsilon(x')\tau$ and in the case of $r_2$ $\varepsilon((C_2)_x)=-\varepsilon(x')\tau$.

Hence the broken line $PEG_2$ can be realised with focal radii $r_1$ or $r_2$ meaning that either
$$
|PE|=t=a-\sqrt{a^2-\mu}-\sqrt{a^2-\nu}+\varepsilon(x')\varepsilon((C_3)_x)\sqrt{a^2-e^2} \mbox{ in the case of } r_1
$$
or
$$
|PE|=t=a-\sqrt{a^2-\mu}+\sqrt{a^2-\nu}-\varepsilon(x')\varepsilon((C_3)_x)\sqrt{a^2-e^2} \mbox{ in the case of } r_2,
$$
respectively. For the distance of $E$ and $G_2$ we have
$$
(EG_2)^2=\left(\rho\frac{\sqrt{a^2-c^2}\sqrt{a^2-e^2}}{\sqrt{a^2-b^2}}-\sqrt{a^2-b^2}\right)^2+ \left(\rho\frac{\sqrt{b^2-c^2}\sqrt{e^2-b^2}}{\sqrt{a^2-b^2}}\right)^2=
$$
$$
=(-e^2)+(a^2-b^2)+(a^2+b^2)+(-c^2)-2\rho\sqrt{(a^2-c^2)(a^2-e^2)}=
$$
$$
=\left(\sqrt{a^2-c^2}-\rho\sqrt{a^2-e^2}\right)^2,
$$
where $\rho $ is $+1$ if $E$ and $G_2$ are in the same half-space with respect to the plane $yz$ and $-1$ in the other case. From this we get that with respect to the fact that $G_2$ and $P$ are the same or opposite half-spaces we get the optimal polygonal lengths:
$$
|PE|+|EG_2|=
a-\sqrt{a^2-\mu}-\varepsilon \sqrt{a^2-\nu}+\sqrt{a^2-c^2},
$$
where $\varepsilon$ is positive or negative with respect to that $\varepsilon(x')$ positive or negative, respectively.

Similarly, the broken line $PHF_1$ can be realised in the cases when we have either consider $r_3$ with
$$
|PH|=\tau=a+\sqrt{a^2-\mu}+\sqrt{a^2-\nu}-\sqrt{a^2-f^2}
$$
or $r_4$ with
$$
|PH|=\tau=a+\sqrt{a^2-\mu}-\sqrt{a^2-\nu}-\sqrt{a^2-f^2},
$$
respectively.

Since we have that
$$
|HF_1|=\sqrt{\left(\varepsilon((C_2)_{\bar{x}})\frac{\sqrt{a^2-b^2}\sqrt{a^2-f^2}}{\sqrt{a^2-c^2}}+\sqrt{a^2-c^2}\right)^2+ (\varepsilon((C_2)_z)\frac{\sqrt{b^2-c^2}\sqrt{f^2-b^2}}{\sqrt{a^2-c^2}})^2}=
$$
$$
=\sqrt{a^2-b^2}+\varepsilon((C_2)_{\bar{x}})\sqrt{a^2-f^2},
$$
and we get for the polygonal length of $PHF_1$
$$
|PH|+|HF_1|= a+\sqrt{a^2-\mu}+\varepsilon\sqrt{a^2-\nu}+\sqrt{a^2-b^2}
$$
if $H$ and $F_1$ corresponding to distinct branches of the focal hyperbola and $\varepsilon$ is positive or negative with respect to that $P$ and $F_1$ are on the same or opposite halfspaces of the $yz$ plane.
If $H$ and $F_1$ corresponding to the same branch of the focal hyperbola we have
$$
|PH|+|HF_1|=a+\sqrt{a^2-\mu}-\varepsilon \sqrt{a^2-\nu}-\sqrt{a^2-b^2},
$$
where again $\varepsilon=1$ if and only if $P$ and $F_1$ are on the same half-space of the $yz$ plane.

Comparing the above results with the assumption of the statement we get that the optimal polygonal length is
$$
|PE|+|EG_2|+|PH|+|HF_1|=
2a+\sqrt{a^2-c^2}-\sqrt{a^2-b^2},
$$
as we stated.
\end{proof}

\end{document}